\documentclass[a4paper,12pt,english]{amsart}
\usepackage{amsmath,latexsym,amssymb,amsfonts}
\usepackage{amscd,amssymb,epsfig}
\usepackage{hyperref}
\usepackage{mathdots}
\usepackage{enumerate}
\usepackage{IEEEtrantools}

\numberwithin{equation}{section}

\newtheorem{theorem}{Theorem}[section]
\newtheorem{lemma}[theorem]{Lemma}
\newtheorem{corollary}[theorem]{Corollary}
\newtheorem{proposition}[theorem]{Proposition}

\theoremstyle{definition}
\newtheorem{definition}[theorem]{Definition}
\newtheorem{remark}[theorem]{Remark}

\renewcommand\labelenumi{(\roman{enumi})}
\renewcommand\theenumi\labelenumi

\renewcommand{\epsilon}{\varepsilon}

\newcommand{\B}{\mathbb{B}}
\newcommand{\C}{\mathbb{C}}
\newcommand{\D}{\mathbb{D}}
\newcommand{\E}{\mathbb{E}}
\newcommand{\F}{\mathcal{F}}
\newcommand{\G}{\mathcal{G}}
\newcommand{\bH}{\mathbb{H}}
\newcommand{\cL}{\mathcal{L}}
\newcommand{\N}{\mathbb{N}}
\renewcommand{\P}{\mathcal{P}}
\newcommand{\R}{\mathbb{R}}
\renewcommand{\S}{\mathcal{S}}
\newcommand{\U}{\mathbb{U}}

\newcommand{\id}{\operatorname{id}}
\newcommand{\tr}{\operatorname{tr}}
\newcommand{\Tr}{\operatorname{Tr}}
\renewcommand{\Re}{\operatorname{Re}}
\renewcommand{\Im}{\operatorname{Im}}
\newcommand{\sa}{\mathrm{sa}}
\newcommand{\1}{\mathbf{1}}
\newcommand{\Prob}{\operatorname{Prob}}

\renewcommand{\dot}[1] {\overset{\,_{\mbox{\bf .}}}{#1}}

\makeindex

\title[The Dyson equation for $2$-positive maps and densities of states]{The Dyson equation for $2$-positive maps and H\"older bounds for the L\'evy distance of densities of states}

\author[T. Mai]{Tobias Mai}
\address{Saarland University, Department of Mathematics, D-66123 Saarbr\"ucken, Germany}
\email{mai@math.uni-sb.de}

\date{\today}

\thanks{The author wants to thank Johannes Hoffmann and Roland Speicher for the interesting discussions which have inspired this work.}



\begin{document}

\begin{abstract}
The so-called density of states is a Borel probability measure on the real line associated with the solution of the Dyson equation which we set up, on any fixed $C^\ast$-probability space, for a selfadjoint offset and a $2$-positive linear map. Using techniques from free noncommutative function theory, we prove explicit H\"older bounds for the L\'evy distance of two such measures when any of the two parameters varies.
As the main tools for the proof, which are also of independent interest, we show that solutions of the Dyson equation have strong analytic properties and evolve along any $C^1$-path of $2$-positive linear maps according to an operator-valued version of the inviscid Burgers equation.
\end{abstract}

\maketitle

\section{Introduction and presentation of the main results}

Voiculescu's groundbreaking discovery \cite{Voi1991} of the phenomenon of asymptotic freeness has started an enormously fruitful exchange of techniques between the area of operator algebras and random matrix theory through free probability theory.
Undoubtedly among the most prominent of such connections which were facilitated by free probability theory are operator-valued semicircular elements. Let us recall that those are noncommutative random variables which constitute the canonical generalization of semicircular elements to the framework of operator-valued free probability theory, while the latter are the free probability counterpart of normally distributed random variables in classical probability; we refer to \cite{Speicher1998,NicaSpeicher2006,MS2017} for an introduction to (operator-valued) free probability theory and to operator-valued semicircular elements in particular.
These are fundamental objects in free probability \cite{Voi1995,Speicher1998,Shlyakhtenko1999,Shlyakhtenko2000,BPV2013,Soltanifar2014}, which fully reveal their rich structure at the frontiers of that field.
On the one hand, the study of random matrix models for operator-valued semicircular elements has led to spectacular developments on the operator-algebraic side like in \cite{Dykema1994,Radulescu1994,Shlyakhtenko1999} and \cite{HT2005,HST2006}. On the other hand, random matrix models for operator-valued semicircular elements are of great interest beyond operator-algebraic applications and this elegant description of their asymptotic eigenvalue distribution enabled a detailed study thereof like in \cite{Shlyakhtenko1996,HRFS2007,RFOBS2008,AjEK2019,AEKN2019,AEK2020}; see also \cite{SS2015,MSY2018,MSY2019,BM2021,CMMPY2021}.

Operator-valued semicircular elements often do not show up explicitly in such contexts, but rather in the guise of their operator-valued Cauchy transforms or by equations determining the latter. To build our intuition, let us take a closer look at the model considered in \cite{HT2005,HST2006}. Voiculescu's multivariate generalization \cite{Voi1991} of Wigner's semicircle law \cite{Wigner1955} (in its almost sure version \cite{HP2000,T2000} which generalizes \cite{Arnold1967} to the multivariate case) teaches us that tuples $(X^N_1,\dots,X^N_n)$ of independent standard Gaussian random matrices of size $N \times N$ converge in distribution almost surely as $N\to \infty$ towards a tuple $(S_1,\dots,S_n)$ of freely independent standard semicircular elements. Explicitly, with respect to the normalized trace $\tr_N$ on $M_N(\C)$, we have for every $k\in\N$ and each choice of indices $1\leq i_1,\dots,i_k \leq n$ that
\begin{equation}\label{eq:GUE_convergence}
\lim_{N\to \infty} \tr_N(X_{i_1}^N X_{i_2}^N \cdots X_{i_k}^N) = \tau(S_{i_1} S_{i_2} \cdots S_{i_k}) \qquad\text{almost surely},
\end{equation}
where we suppose that $(S_1,\dots,S_n)$ is realized in some $C^\ast$-probability space $(M,\tau)$. 
For fixed selfadjoint matrices $b_0,b_1,\dots,b_n\in M_m(\C)$, we aim at understanding the asymptotic behavior of the selfadjoint random matrices $X^N := b_0 \otimes \1_N + b_1 \otimes X_1^N + \dots + b_n \otimes X_n^N$. It follows from \eqref{eq:GUE_convergence} that the empirical eigenvalue distribution $\mu_{X^N}$ of $X^N$ converges in distribution almost surely as $N\to \infty$ to the analytic distribution $\mu_S$ of the noncommutative random variable $S := b_0 \otimes \1 + b_1 \otimes S_1 + \dots + b_n \otimes S_n$ in the $C^\ast$-probability space $(M_m(\C) \otimes M, \tr_m \otimes \tau)$. In fact, by means of operator-valued free probability, we can even compute $\mu_S$. The crucial observation is that $S$, in the operator-valued $C^\ast$-probability space $(M_m(\C) \otimes M, \E, M_m(\C))$ with conditional expectation $\E := \id_{M_m(\C)} \otimes \tau$, forms an operator-valued semicircular element with mean $b_0$ and the covariance $\eta: M_m(\C) \to M_m(\C)$ given by $\eta(b) := \sum^n_{j=1} b_j b b_j$ for $b\in M_m(\C)$. This implies that the operator-valued Cauchy transform $G_S: \bH^+(M_m(\C)) \to \bH^-(M_m(\C))$ defined by $G_S(b) := \E[(b-S)^{-1}]$ solves the equation
$$(b - b_0) G_S(b) = \1 + \eta(G_S(b)) G_S(b) \qquad\text{for $b\in \bH^+(M_m(\C))$}.$$
In fact, $G_S$ is hereby uniquely determined and can be computed efficiently by some fixed point iteration; see \cite{HRFS2007,HMS2021}. This change of perspective thus settles the problem of determining $\mu_S$, since $G_S$ recovers the scalar-valued Cauchy transform $\G_{\mu_S}$ of $\mu_S$ via the relation $\G_{\mu_S}(z) = \tr_m(G_S(z\1))$ for all $z\in\C^+$ and since $\mu_S$ can be obtained from $\G_{\mu_S}$ via Stieltjes inversion.

In this paper, we are concerned with a significant generalization of this construction which originates in \cite{AjEK2019,AEK2020} and which we briefly sketch now; a more detailed presentation will be given in Section \ref{sec:Dyson-density_of_states}, in particular in Definition \ref{def:density_of_states}.
Let $(B,\phi)$ be a $C^\ast$-probability space. We put $B_\sa := \{b \in B \mid b^\ast = b\}$ and denote by $\P(B)$ the set of all $\C$-linear maps $\eta: B \to B$ which are positive; see Subsection \ref{subsec:positive_maps} for more details.
To any $\rho = (b_0,\eta) \in B_\sa \times \P(B)$, called \emph{data pair}, we associate the holomorphic function $G_\eta: \bH^+(B) \to \bH^-(B)$ which is uniquely determined by the \emph{Dyson equation}
\begin{equation}\label{eq:Dyson_equation}
b G_\eta(b) = \1 + \eta(G_\eta(b)) G_\eta(b) \qquad\text{for $b\in \bH^+(B)$}
\end{equation}
and consider in turn the holomorphic function $\G_\rho: \C^+ \to \C^-$ which is defined by $\G_\rho(z) := \phi(G_\eta(z \1 - b_0))$ for all $z\in \C^+$. Since $\G_\rho$ satisfies $\lim_{y\to \infty} iy\G_\rho(i y) = 1$, it must be the Cauchy transform of a hereby uniquely determined Borel probability measure $\mu_\rho$ on $\R$, to which we shall refer as the \emph{density of states} associated with $\rho$. In this way, we obtain a mapping
$$\mu:\ B_\sa \times \P(B) \to \Prob(\R),\quad \rho \mapsto \mu_\rho$$
taking values in the set $\Prob(\R)$ of all Borel probability measures on the real line $\R$.

For the operator-valued semicircular element $S$ with mean $b_0$ and covariance $\eta$ which we introduced above, this procedure recovers $\mu_S$ as $\mu_{(b_0,\eta)}$; in fact, it works for any operator-valued semicircular element $S$ in an arbitrary operator-valued $C^\ast$-probability space $(A,\E,B)$ and with respect to some distinguished state $\phi$ on $B$. This highlights the fundamental fact that, if we keep $(B,\phi)$ fixed, then $\mu_S$ only depends on $b_0$ and the completely positive linear map $\eta$, but not on the concrete realization $S$ or the corresponding space $(A,\E,B)$.
We point out that with \cite[Proposition 2.2]{PV2013}, it is possible to construct for every data pair $(b_0,\eta)$ with $\eta$ being completely positive an operator-valued $C^\ast$-probability space $(A,\E,B)$ and an operator-valued semicircular element $S$ with mean $b_0$ and covariance $\eta$. The map $\mu$ introduced above not only circumvents this construction, but even works in situations where no such realization of $(b_0,\eta)$ can exist.

Our aim is to quantify the proximity of $\mu_{\rho_0}$ and $\mu_{\rho_1}$ in terms of the data pairs $\rho_0,\rho_1 \in B_\sa \times \P(B)$. To this end, we endow
\begin{itemize}
 \item $\Prob(\R)$ with the L\'evy metric $L: \Prob(\R) \times \Prob(\R) \to [0,1]$, which is known to provide a metrization of weak convergence;
 \item $B_\sa$ with the metric $d_{B_\sa}$ induced by the norm $\|\cdot\|$ on the $C^\ast$-algebra $B$; 
 \item $\P(B)$ with the metric $d_{\P(B)}$ induced by the norm $\|\eta\| := \sup_{\|b\|\leq 1} \|\eta(b)\|$ on the space of all bounded linear operators $\eta: B \to B$.
\end{itemize}
More precisely, we will derive H\"older bounds for $L(\mu_{\rho_1},\mu_{\rho_0})$ in terms of $d_{B_\sa}(b_{0,1},b_{0,0})$ and $d_{\P(B)}(\eta_1,\eta_0)$, depending on the varying argument, but uniformly in the respective other argument; for that purpose, however, we need to restrict $\mu$ to the subset $B_\sa \times \P_2(B)$ of $B_\sa \times \P(B)$, where $\P_2(B)$ stands for the set of all $\eta \in \P(B)$ which are $2$-positive. The necessity of restricting $\mu$ arose also in \cite{AEK2020}, but there the authors required $\eta$ to satisfy some strict version of positivity called \emph{flatness}. This point and that we drop any symmetry condition for $\eta$ are the crucial differences between our approach and that of \cite{AEK2020}, aside from the technical variation that we work in the more general setting of $C^\ast$-probability spaces which are not necessarily tracial, whereas \cite{AEK2020} is formulated for tracial $W^\ast$-probability spaces.

For the sake of simplicity, we will denote the restriction of $\mu$ to $B_\sa \times \P_2(B)$ again by $\mu$. The precise statements, which constitute the first main result of this paper, read as follows. 

\begin{theorem}\label{thm:Levy_Hoelder}
Let $(B,\phi)$ be a $C^\ast$-probability space and consider the map $\mu: B_\sa \times \P_2(B) \to \Prob(\R)$ defined above. Then the following bounds hold true:
\begin{enumerate}
 \item\label{it:Levy_Hoelder-1} With $c_1 := 3 (\frac{1}{\pi})^{1/3} < 2.0484$, we have for each choice of $\eta \in \P_2(B)$ that
$$L(\mu_{(b_{0,1},\eta)},\mu_{(b_{0,0},\eta)}) \leq c_1\, \|b_{0,1} - b_{0,0}\|^{1/3} \qquad\text{for all $b_{0,0},b_{0,1} \in B_\sa$.}$$
 \item\label{it:Levy_Hoelder-2} With $c_2 := 5 (\frac{1}{4\pi})^{2/5} < 1.8168$, we have for each choice of $b_0 \in B_\sa$ that
$$L(\mu_{(b_0,\eta_1)},\mu_{(b_0,\eta_0)}) \leq c_2\, \|\eta_1 - \eta_0\|^{1/5} \qquad\text{for all $\eta_0,\eta_1\in\P_2(B)$.}$$
\end{enumerate}
In particular, if the Cartesian product $B_\sa \times \P_2(B)$ is endowed with any product metric $d$, then the map $\mu: (B_\sa \times \P_2(B),d) \to (\Prob(\R),L)$ is uniformly continuous.
\end{theorem}

The proof of Theorem \ref{thm:Levy_Hoelder} will be given in Section \ref{sec:Levy_Hoelder}. For both of the H\"older bounds \ref{it:Levy_Hoelder-1} and \ref{it:Levy_Hoelder-2}, we will consider an affine linear interpolation $(\rho_t)_{t\in[0,1]}$ between $\rho_0$ and $\rho_1$ for which the ``time derivative'' of $\G_{\rho_t}$ can be controlled. Along the way, we obtain estimates for the integrals $\frac{1}{\pi} \int^\infty_{-\infty} |\G_{\mu_{\rho_1}}(s+i\epsilon) - \G_{\mu_{\rho_0}}(s+i\epsilon)|\, ds$ for $\epsilon>0$ which we record in Corollary \ref{cor:integral_bounds}.

Theorem \ref{thm:Levy_Hoelder} can be seen as an extension of \cite[Proposition 2.3]{AEK2020} to the less restrictive setting of covariance maps that are $2$-positive, where the proximity of two densities of states is controlled ``globally'' in terms of the L\'evy distance. Due to its relaxed conditions, our result applies readily to the limit laws of Kronecker random matrices as discussed in \cite[Section 9]{AEK2020}; see also \cite{AEKN2019}.

The bounds in Theorem \ref{thm:Levy_Hoelder} \ref{it:Levy_Hoelder-2} and Corollary \ref{cor:integral_bounds} \ref{it:integral_bounds-2} appeared already in \cite[Theorem 3.5]{BM2021}, but only in the centered case and under much stronger assumptions regarding $\eta_0$ and $\eta_1$. While the proof given in \cite{BM2021} used some ``discrete interpolation'' stemming from the (noncommutative) Lindeberg method, we will obtain these results via some ``continuous interpolation'' as explained above.
For that purpose, it is crucial to understand how solutions $G_\eta$ of the Dyson equation \eqref{eq:Dyson_equation} depend on $\eta \in \P_2(B)$; this is the content of the next theorem, which is our second main result. Notice that $\overline{\bH^+_\gamma(B)} = \{b\in B \mid \Im(b) \geq \gamma \1\}$ for $\gamma>0$.

\begin{theorem}\label{thm:continuity}
Let $G: \P_2(B) \times \bH^+(B) \to \bH^-(B), (\eta,b) \mapsto G_\eta(b)$ be the function determined by the Dyson equation \eqref{eq:Dyson_equation}. We endow $\P_2(B) \times \bH^+(B)$ with the product topology. Then:
\begin{enumerate}
 \item\label{it:continuity-0} For every fixed $\gamma>0$, the family $G_{\boldsymbol{\cdot}}(b): \P_2(B) \to \bH^-(B), \eta \mapsto G_\eta(b)$ with $b\in \overline{\bH^+_\gamma(B)}$ is equicontinuous, i.e., we have
\begin{equation}\label{eq:continuity}
\lim_{\eta\to \eta_0} \sup_{b\in \overline{\bH^+_\gamma(B)}} \|G_\eta(b) - G_{\eta_0}(b)\| = 0 \qquad\text{for each $\eta_0\in \P_2(B)$}.
\end{equation}
In particular, $G: \P_2(B) \times \bH^+(B) \to \bH^-(B), (\eta,b) \mapsto G_\eta(b)$ is continuous.
 \item\label{it:continuity-1} For every fixed $\gamma>0$, the family $(D G_{\boldsymbol{\cdot}})(b): \P_2(B) \to \cL(B), \eta \mapsto (DG_\eta)(b)$ with $b\in \overline{\bH^+_\gamma(B)}$ is equicontinuous, i.e., we have
\begin{equation}\label{eq:continuity_derivative}
\lim_{\eta\to \eta_0} \sup_{b\in \overline{\bH^+_\gamma(B)}} \|(DG_\eta)(b) - (DG_{\eta_0})(b)\| = 0 \qquad\text{for each $\eta_0\in \P_2(B)$}.
\end{equation}
In particular, $DG: \P_2(B) \times \bH^+(B) \to \cL(B), (\eta,b) \mapsto (DG_\eta)(b)$ is continuous.
\end{enumerate}
\end{theorem}

The proof of Theorem \ref{thm:continuity} will be given in Section \ref{sec:Burgers_equation}.

Our third main result given in the following theorem says that the evolution of solutions of the Dyson equation \eqref{eq:Dyson_equation} along $C^1$-paths $\eta: [0,T] \to \P_2(B)$ can be described by an operator-valued version of the inviscid Burgers equation \eqref{eq:Burgers}.

\begin{theorem}\label{thm:Burgers}
Let $\eta: [0,T] \to \P_2(B), t\mapsto \eta_t$ be a $C^1$-path. Let $G: [0,T] \times \bH^+(B) \to \bH^-(B), (t,b) \mapsto G_t(b)$ be the function defined by $G_t(b) := G_{\eta_t}(b)$ for all $t \in [0,T]$ and $b\in \bH^+(B)$. We endow $[0,T] \times \bH^+(B)$ with the product topology. Then, for every fixed $\gamma>0$, the family $G_{\boldsymbol{\cdot}}(b): [0,T] \to \bH^-(B)$ with $b\in \overline{\bH^+_\gamma(B)}$ is equidifferentiable, i.e., we have
\begin{equation}\label{eq:differentiability}
\lim_{t\to t_0} \sup_{b\in \overline{\bH^+_\gamma(B)}} \Big\|\frac{1}{t-t_0} \big(G_{t}(b) - G_{t_0}(b)\big) - \dot{G}_{t_0}(b)\Big\| = 0 \qquad\text{for each $t_0\in [0,T]$}.
\end{equation}
Moreover, $G$ satisfies the partial differential equation
\begin{equation}\label{eq:Burgers}
\dot{G}_t(b) = - (D G_t)(b)\big(\dot{\eta}_t(G_t(b))\big) \qquad\text{for $(t,b) \in [0,T] \times \bH^+(B)$}.
\end{equation}
In particular, $\dot{G}: [0,T] \times \bH^+(B) \to B, (t,b) \mapsto \dot{G}_t(b)$ is continuous and $\dot{G}_t: \bH^+(B) \to B$ is holomorphic for every fixed $t\in[0,T]$.
\end{theorem}

In the scalar-valued framework of \cite{Voi1986}, Voiculescu derived the inviscid Burgers equation as the free probability counterpart of the heat equation. An operator-valued generalization of his result (see also \cite{Voi1995}) was obtained by Anshelevich, Belinschi, Fevrier, and Nica in \cite{ABFN2013}, but under more restrictive assumptions. We prove Theorem \ref{thm:Burgers} in Section \ref{sec:Burgers_equation}.

The rest of this paper is organized as follows.

In Section \ref{sec:preliminaries}, we set up the notation and we recall some fundamental results from different areas that will be used in the subsequent parts; concretely, Section \ref{subsec:positive_maps} is about positive and completely positive maps on $C^\ast$-algebras, Section \ref{subsec:holomorphy} gives a brief survey on holomorphic functions between (subsets of) Banach spaces, Section \ref{subsec:Cauchy_transforms} focuses on Cauchy transforms and their characterization, and Section \ref{subsec:measure_distances} recalls the definition of the L\'evy distance and explains how this quantity can be controlled in terms of the respective Cauchy transforms.

Section \ref{sec:Dyson-density_of_states} has two purposes. On the one hand, we want to describe the construction of the density of states $\mu: B_\sa \times \P(B) \to \Prob(\R)$, which is the main object of this paper. In doing so, we follow \cite{AEK2020} but we generalize their construction to $C^\ast$-probability spaces $(B,\phi)$, where $\phi$ does not need to be tracial, and to positive maps $\eta$ defined on $B$, which are not necessarily selfadjoint with respect to the distinguished state $\phi$. The definition of the density of states $\mu: B_\sa \times \P(B) \to \Prob(\R)$ will be given in Section \ref{subsec:density_of_states}, while Section \ref{subsec:Dyson} provides the results about the Dyson equation which justify this construction; as for the latter, we build on \cite{HRFS2007} but we add some explanation why solutions of the Dyson equation are holomorphic (and in fact uniformly Fr\'echet differentiable on $\overline{\bH^+_\gamma(B)}$ for every $\gamma>0$ as shown in Lemma \ref{lem:Taylor_bound_Cauchy}).
On the other hand, we undertake in Section \ref{sec:Dyson-density_of_states} a detailed study of properties of solutions of the Dyson equation, especially in the case of $2$-positive maps $\eta$. To this end, we consider in Section \ref{subsec:Dyson_amplification} amplifications of the Dyson equation. More precisely, we show in Proposition \ref{prop:opval_semicircular_Cauchy_matrix} that for an $n$-positive map $\eta: B \to B$ the solutions $(G_{\eta^{(k)}})_{k=1}^n$ for the amplifications $\eta^{(k)}: M_k(B) \to M_k(B)$ for $k=1,\dots,n$ glue to a noncommutative function of height $n$; this terminology is outlined in the Appendix \ref{sec:FreeNCFunctionTheory}. In Section \ref{subsec:Dyson_2-positive}, we focus on $2$-positive linear maps $\eta$ and derive in Lemma \ref{lem:Cauchy_derivatives_bound} with the help of Proposition \ref{prop:opval_semicircular_Cauchy_matrix} bounds for the Fr\'echet derivative $D G_\eta$ of $G_\eta$, which enable us to establish Lipschitz continuity of $G_\eta: \bH^+(B) \to \bH^-(B)$ and of $D G_\eta: \bH^+(B) \to \cL(B)$ on each $\overline{\bH^+_\gamma(B)}$ for $\gamma>0$; see Lemma \ref{lem:Lipschitz_bound} and Lemma \ref{lem:Lipschitz_bound_derivative}.

Section \ref{sec:Burgers_equation} is devoted to the proofs of Theorem \ref{thm:continuity} and of Theorem \ref{thm:Burgers}. Finally, with the help of these results, the proof of Theorem \ref{thm:Levy_Hoelder} will be given in Section \ref{sec:Levy_Hoelder}.

\section{Preliminaries}\label{sec:preliminaries}

\subsection{Positive and completely positive maps}\label{subsec:positive_maps}

For complex Banach spaces $(E,\|\cdot\|_E)$ and $(F,\|\cdot\|_F)$, we denote by $\cL(E,F)$ the Banach space of all bounded linear operators $T$ from $E$ to $F$ with the norm given by $\|T\| := \sup_{\|x\|_E \leq 1} \|T x\|_F$. We abbreviate $\cL(E) := \cL(E,E)$. 

For a unital $C^\ast$-algebra $B$, whose unit we denote by $\1$, we let $B_\sa := \{b\in B \mid b^\ast = b\}$ be the real subspace of $B$ consisting of all selfadjoint elements and we let $B_+ := \{b\in B \mid b \geq 0\}$ be the convex cone of all positive elements in $B$.
For every $k\in\N$, we denote by $M_k(B)$ the unital $C^\ast$-algebra of all $k \times k$ matrices over $B$; the unit in $M_k(B)$ will be denoted by $\1_k$.

A map $\eta: B \to B$ is said to be \emph{positive} if it is $\C$-linear and maps $B_+$ into itself. Note that every such $\eta$ is automatically bounded, i.e., it belongs to $\cL(B)$; in fact, we have $\|\eta\| = \|\eta(\1)\|$ (see, for instance, Corollary 2.9 in \cite{Paulsen2002}). The set of all positive maps $\eta: B \to B$, which will be denoted by $\P(B)$, thus forms a convex subset of $\cL(B)$.
Further, for every $k\in\N$, we denote by $\P_k(B)$ the subset of $\P(B)$ consisting of those $\eta$ which are \emph{$k$-positive}, i.e., for which the amplification $\eta^{(k)}: M_k(B) \to M_k(B)$, given by $\eta^{(k)}(b) := (\eta(b_{ij}))_{i,j=1}^k$ for every matrix $b=(b_{ij})_{i,j=1}^k \in M_k(B)$, are positive in the aforementioned sense. Clearly,
$$\P(B) = \P_1(B) \supseteq \dots \supseteq \P_k(B) \supseteq \P_{k+1}(B) \supseteq \dots \supseteq \P_\infty(B),$$
where $\P_\infty(B)$ denotes the subset of $\P(B)$ consisting of those $\eta \in \P(B)$ which are \emph{completely positive}, i.e., which belong to $\P_k(B)$ for every $k\in\N$; in other words, $\P_\infty(B) := \bigcap_{k\geq 1} \P_k(B)$. Notice that $\P_k(B)$, for every $k\in\N$, and likewise $\P_\infty(B)$ form convex subsets of $\cL(B)$.

As shown in \cite{Choi1972}, each of the inclusions in the above chain can be strict depending on $B$; more precisely, \cite[Theorem 1]{Choi1972} says that we have $\P_{n-1}(M_n(\C)) \supsetneq \P_n(M_n(\C))$. On the other hand, \cite[Theorem 5]{Choi1972} tells us that $\P_n(M_n(\C)) = \P_\infty(M_n(\C))$. In particular, already on $M_3(\C)$ there are $2$-positive maps which are not completely positive; a concrete example that was given in \cite{Choi1972} is $\eta: M_3(\C) \to M_3(\C)$ defined by $\eta(b) := 2 \Tr_3(b) \1_3 - b$. In view of these facts, our results on the operator-valued inviscid Burgers equation provide a significant extension of the previous work done in this direction.

\subsection{Holomorphic functions between Banach spaces}\label{subsec:holomorphy}

In this section, we give a survey on notions of holomorphy for functions between subsets of complex Banach spaces and we briefly recall some of their properties which will be used in the sequel. For more details, we refer the interested readers to the monographs \cite{HillePhillips1974, Dineen1999}.

In the following, let $(E,\|\cdot\|_E)$ and $(F,\|\cdot\|_F)$ be complex Banach spaces.

Recall that a function $f: U\to F$ on an open subset $\emptyset \neq U \subseteq E$ is said to be \emph{G\^ateaux holomorphic on $U$} if the limit, called the \emph{G\^ateaux derivative} of $f$ at $x$ in the direction $h$,
$$\delta f(x;h) := \lim_{\substack{z\rightarrow 0\\ z\in U(x;h)\backslash\{0\}}} \frac{1}{z}(f(x+zh)-f(x))$$
exists for all $x\in U$ and $h\in E$, where $U(x;h) := \{z\in\C \mid x+zh\in U\}$.
According to \cite[Lemma 3.3]{Dineen1999}, a function $f: U \to F$ is G\^ateaux holomorphic if and only if, for every choice of $x\in U$ and $h\in E$ and each bounded linear functional $\phi: F \to \C$, the $\C$-valued function $U(x;h) \ni z \mapsto \phi(f(x+zh))$ is holomorphic in the sense of one-variable complex analysis.

Further, we recall that a function $f: U \to F$ is said to be \emph{Fr\'echet holomorphic on $U$}, if
\begin{enumerate}
\item\label{it:Frechet_Gateaux} it is G\^ateaux holomorphic on $U$,
\item\label{it:Frechet_bounded} its G\^ateaux derivative $\delta f(x;\cdot)$ at any point $x\in U$ is a bounded linear operator, and
\item\label{it:Frechet_limit} $\displaystyle\lim_{h \rightarrow 0} \frac{1}{\|h\|_E} \|f(x+h)-f(x)-\delta f(x;h)\|_F = 0$.
\end{enumerate}
In this case, we write $Df(x)$ for the bounded linear operator given by the G\^ateaux derivative $\delta f(x;\cdot)$ and call it the \emph{Fr\'echet derivative} of $f$ at the point $x\in U$.
It can be shown (see (2.3) in \cite{Zorn1945}) that $\delta f(x;\cdot)$ is always linear; therefore, condition \ref{it:Frechet_bounded} above reduces to the requirement that $\delta f(x;\cdot)$ is bounded. A result of Zorn \cite{Zorn1946} says that condition \ref{it:Frechet_limit} is redundant. 

It is an immediate consequence of the definitions that every Fr\'echet holomorphic function $f: U \to F$ is G\^ateaux holomorphic and continuous. In particular, every Fr\'echet holomorphic function is holomorphic in the following sense (see \cite{Harris2003}): a function $f: U \to F$ is said to be \emph{holomorphic} if it is G\^ateaux holomorphic on $U$ and \emph{locally bounded on $U$}, i.e., for every $x_0\in U$ there exists $r=r(x_0)>0$ such that
$$\sup_{x\in U\colon \|x-x_0\|_E < r} \|f(x)\|_F < \infty.$$
Remarkably, it turns out that conversely every holomorphic function $f: U \to F$ necessarily is Fr\'echet holomorphic; see \cite[Theorem 3.17.1]{HillePhillips1974}. In other words, holomorphy and Fr\'echet holomorphy are equivalent concepts. The following lemma is a slight generalization of some key arguments in the proof of the aforementioned fact, which we record for later use. 
To shorten the notation, we set $\B_E(x_0,r) := \{x\in E \mid \|x-x_0\|_E < r\}$ for $x_0 \in E$ and $r>0$.

\begin{lemma}\label{lem:Cauchy_Taylor_estimates}
Let $f: U \to F$ be a G\^ateaux holomorphic function. Suppose that $x_0 \in U$ and $r>0$ are such that $\B_E(x_0,r) \subseteq U$ and $M_f(x_0,r) := \sup_{x\in \B_E(x_0,r)} \|f(x)\|_F < \infty$. Then the following statements hold true:
\begin{enumerate}
 \item\label{it:Cauchy_estimate} For every $h \in E$, we have that $\|\delta f(x_0; h)\|_F \leq \frac{1}{r} M_f(x_0,r) \|h\|_E$.
 \item\label{it:Taylor_estimate} For $h\in \B_E(0,r)$, we have $\|f(x_0+h) - f(x_0) - \delta f(x_0; h)\|_F \leq \frac{M_f(x_0,r)}{r(r-\|h\|_E)} \|h\|_E^2.$
\end{enumerate}
\end{lemma}

The different characterizations of holomorphy are useful to carry over results from one-variable complex analysis. The next lemma, for instance, is a consequence of Vitali's theorem; it is contained in \cite[Theorem 3.18.1]{HillePhillips1974}, but since the proof simplifies significantly in the particular situation that we are considering here, we include the argument for the sake of clarity.
Recall that a family $\F$ of functions $f: U \to F$ is called \emph{locally uniformly bounded} if for every $x_0\in U$ there exists $r=r(x_0)>0$ such that
$$\sup_{f\in\F}\ \sup_{x\in U\colon \|x-x_0\|_E < r} \|f(x)\|_F < \infty.$$

\begin{lemma}\label{lem:holomorphy_criterion}
Let $(f_n)_{n=1}^\infty$ be a sequence of G\^ateaux holomorphic functions $f_n: U \to F$ which is locally uniformly bounded and pointwise convergent to a function $f: U \to F$. Then $f$ is (Fr\'echet) holomorphic.
\end{lemma}

\begin{proof}
Let us fix $x\in U$, $h\in E$, as well as some bounded linear functional $\phi: F \to \C$. Since each $f_n$ is G\^ateaux holomorphic on $U$, we obtain holomorphic functions $g_n: U(x;h) \to \C$ by $g_n(z) := \phi(f_n(x+zh))$; since $(f_n)_{n=1}^\infty$ is locally uniformly bounded on $U$ and converges pointwise on $U$ to $f$, we infer that $(g_n)_{n=1}^\infty$ is locally uniformly bounded on $U(x;h)$ and pointwise convergent on $U(x;h)$ to the function $g: U(x;h) \to \C$ which is given by $g(z) := \phi(f(x+zh))$. By Vitali's theorem, it follows that $(g_n)_{n=1}^\infty$ converges locally uniformly to $g$, so that $g$ must be holomorphic. This verifies that $f$ is G\^ateaux holomorphic on $U$. Because $(f_n)_{n=1}^\infty$ is locally uniformly bounded and pointwise convergent to $f$ by assumption, it follows that $f$ is locally bounded. In summary, we have that $f$ is (Fr\'echet) holomorphic.
\end{proof}

\subsection{Cauchy transforms}\label{subsec:Cauchy_transforms}

We denote by $\Prob(\R)$ the set of all Borel probability measures on the real line $\R$. For every $\mu \in \Prob(\R)$, let $\G_\mu$ be the \emph{Cauchy transform} of $\mu$, i.e., the holomorphic function
$$\G_\mu:\ \C^+ \to \C^-, \quad z\mapsto \int_\R \frac{1}{z-t}\, d\mu(t)$$
which is defined on the complex upper half-plane $\C^+ := \{z\in\C \mid \Im(z) > 0\}$ and takes its values in the complex lower half-plane $\C^- := \{z \in \C \mid \Im(z) < 0\}$.

For later use, we recall the following well-known characterization of Cauchy transforms. We refer the readers to \cite{GH2003} for an excellent survey on this subject; in particular, the statement given below is \cite[Lemma 2]{GH2003}, except that the authors work with the \emph{Stieltjes transform} $\S_\mu: \C^+ \to \C^+$ instead of the Cauchy transform $\G_\mu: \C^+ \to \C^-$, which differ only by sign.

\begin{theorem}\label{thm:Cauchy_characterization}
Let $\G: \C^+ \to \C$ be holomorphic. The following statements are equivalent:
\begin{enumerate}
 \item\label{it:Cauchy_characterization_measure} $\G$ is the Cauchy transform of a Borel probability measure on $\R$, i.e., there exists some $\mu \in \Prob(\R)$ such that $\G(z) = \G_\mu(z)$ for all $z\in\C^+$.
 \item $\G$ satisfies the conditions $\Im(\G(z)) > 0$ for all $z\in \C^+$ and $\lim_{y\to \infty} iy \G(iy) = 1$.
\end{enumerate}
\end{theorem}

We notice that $\mu \in \Prob(\R)$ is in fact uniquely determined by its Cauchy transform $\G_\mu$; in other words (see \cite[Lemma 1]{GH2003}), for $\mu_1,\mu_2 \in \Prob(\R)$, we have that $\G_{\mu_1}(z) = \G_{\mu_2}(z)$ for all $z\in \C^+$ if and only if $\mu_1 = \mu_2$. In particular, we see that the measure $\mu$ in Item \ref{it:Cauchy_characterization_measure} of Theorem \ref{thm:Cauchy_characterization} is unique.
  
\subsection{L\'evy and Kolmogorov distance}\label{subsec:measure_distances}

The L\'evy and the Kolmogorov distance yield metrics on $\Prob(\R)$ which are among the most prominent ones which are usually considered on that space; see, e.g., \cite{Bobkov2016}. In this paper, we provide bounds for the L\'evy distance of densities of states, but we point out that such bounds often go over to the Kolmogorov distance.

Before giving the definitions of the L\'evy and the Kolmogorov distance, we stipulate that $\F_\mu: \R \to [0,1]$ will stand for the \emph{cumulative distribution function} of any $\mu \in \Prob(\R)$; recall that $\F_\mu$ is defined by $\F_\mu(t) := \mu((-\infty,t])$ for every $t\in\R$.

If $\mu,\nu \in \Prob(\R)$ are given, then their \emph{L\'evy distance} is defined by
$$L(\mu,\nu) := \inf\{\epsilon>0 \mid \forall t\in\R:\ \F_\mu(t-\epsilon) - \epsilon \leq \F_\nu(t) \leq \F_\mu(t+\epsilon) + \epsilon\}$$
and their \emph{Kolmogorov distance} is defined by
$$\Delta(\mu,\nu) := \sup_{t\in\R} |\F_\mu(t) - \F_\nu(t)|.$$

It is well-known that the L\'evy distance defines a metric on $\Prob(\R)$ which provides a metrization of weak convergence in the sense that $L(\mu_n,\mu) \to 0$ if and only if $(\mu_n)_{n=1}^\infty$ converges weakly to $\mu$.

The L\'evy distance is always upper-bounded by the Kolmogorov distance, i.e., we have $L(\mu,\nu) \leq \Delta(\mu,\nu)$ for all $\mu,\nu \in \Prob(\R)$. If we suppose in addition that $\nu$ has a cumulative distribution function $\F_\nu$ which is H\"older continuous with exponent $\beta\in(0,1]$ and H\"older constant $C>0$, i.e., $|\F_\nu(t) - \F_\nu(s)| \leq C |t-s|^\beta$ holds for all $s,t\in\R$, then \cite[Lemma B.18]{BS10} says that we also have the complementary bound $\Delta(\mu,\nu) \leq (C+1) L(\mu,\nu)^\beta$.

Both the L\'evy and the Kolmogorov distance can be controlled in terms of the Cauchy transforms of the involved measures. For the Kolmogorov distance, such bounds are given in \cite[Theorem 2.2]{Bai1993a}; see also \cite{Bai1993b,BS10}. These bounds, however, require some regularity of the cumulative distribution function of at least one of the considered measures; see \cite{GT2003,BM2020}. For the L\'evy distance, we have the following statement, by which this quantity becomes accessible via Cauchy transforms and which remarkably works without any constraints.

\begin{theorem}\label{thm:Levy_distance}
Let $\mu,\nu \in \Prob(\R)$. Then, for any choice of $\epsilon>0$, we have the bound
\begin{equation}\label{eq:Levy_bound}
L(\mu,\nu) \leq 2\sqrt{\frac{\epsilon}{\pi}} + \frac{1}{\pi} \int^\infty_{-\infty} \big| \Im(\G_\mu(t+i\epsilon)) - \Im(\G_\nu(t+i\epsilon)) \big|\, dt.
\end{equation}
\end{theorem}

While this result seems to be well-known, it is surprisingly challenging to localize it in the literature. The reason is that bounds for $L(\mu,\nu)$ are mostly given in terms of the characteristic functions of $\mu$ and $\nu$ rather than in terms of their Cauchy transforms; see \cite[Section 3]{Bobkov2016}. A slightly weaker version of Theorem \ref{thm:Levy_distance}, namely with an additional factor $\sqrt{2}$ in front of the term $2\sqrt{\frac{\epsilon}{\pi}}$, was obtained recently in \cite{Salazar2022}. A direct proof of Theorem \ref{thm:Levy_distance} is given in the appendix of \cite{BM2021}. For convenience of the readers, we include here another derivation based on the so-called \emph{smoothing inequality}
\begin{equation}\label{eq:smoothing_inequality}
L(\mu,\nu) \leq L(\mu \ast \gamma, \nu \ast \gamma) + \max\{2\delta, 1 - \F_\gamma(\delta) + \F_\gamma(-\delta)\}
\end{equation}
holding for arbitrary $\mu,\nu,\gamma\in\Prob(\R)$ and each $\delta>0$; a proof of \eqref{eq:smoothing_inequality} can be found in \cite{Zolotarev1973}.

\begin{proof}[Proof of Theorem \ref{thm:Levy_distance}]
For $\mu\in \Prob(\R)$ and $\epsilon>0$, we define $\mu_\epsilon\in\Prob(\R)$ by $d\mu_\epsilon(t) = -\frac{1}{\pi} \Im(\G_\mu(t+i\epsilon))\, dt$. Note that if we let $\gamma_\epsilon \in \Prob(\R)$ be given by $d\gamma_\epsilon(t) := \frac{1}{\pi}\frac{\epsilon}{\epsilon^2+t^2}\, dt$, then $\mu_\epsilon = \mu \ast \gamma_\epsilon$ for each $\epsilon>0$. Now, since we have
$$\F_{\gamma_\epsilon}(-\delta) = \gamma_\epsilon((-\infty,-\delta]) = \frac{1}{\pi} \int^{-\delta}_{-\infty} \frac{\epsilon}{\epsilon^2+t^2}\, dt \leq \frac{\epsilon}{\pi\delta}$$
and similarly $1 - \F_{\gamma_\epsilon}(\delta) = \gamma_\epsilon((\delta,\infty)) \leq \frac{\epsilon}{\pi\delta}$, the smoothing inequality \eqref{eq:smoothing_inequality} yields $L(\mu,\nu) \leq L(\mu_\epsilon,\nu_\epsilon) + 2\max\{\delta,\frac{\epsilon}{\pi\delta}\}$. By choosing $\delta = \sqrt{\frac{\epsilon}{\pi}}$, we obtain that $L(\mu,\nu) \leq L(\mu_\epsilon,\nu_\epsilon) + 2\sqrt{\frac{\epsilon}{\pi}}$. In combination with the fact that
$$L(\mu_\epsilon,\nu_\epsilon) \leq \frac{1}{\pi} \int^\infty_{-\infty} \big|\Im(\G_\mu(t+i\epsilon)) - \Im(\G_\nu(t+i\epsilon))\big|\, dt,$$
which is easily verified by straightforward estimates, this yields the asserted bound \eqref{eq:Levy_bound}.
\end{proof}

\section{The Dyson equation and density of states}\label{sec:Dyson-density_of_states}

Throughout the following, let $B$ be a unital $C^\ast$-algebra; the unit of $B$ will be denoted by $\1$. We define $\bH^+(B)$ and $\bH^-(B)$ to be the upper and lower half-plane in $B$, respectively, that is, $\bH^\pm(B) := \{b\in B \mid \exists \epsilon>0: \pm \Im(b) \geq \epsilon \1\}$ where $\Im(b) := \frac{1}{2i} (b - b^\ast)$.
Notice that $b\in B$ belongs to $\bH^\pm(B)$ if and only if $\pm \Im(b)$ is an invertible positive element in $B$; recall that $b \geq \|b^{-1}\|^{-1} \1$ for every positive $b\in B$ which is invertible.
Further, for every $\gamma>0$, we define $\bH^+_\gamma(B) := \{b + i\gamma\1 \mid b \in \bH^+(B)\}$. Notice that $\overline{\bH^+_\gamma(B)} = \{b \in \bH^+(B) \mid \Im(b) \geq \gamma \1\}$.

\subsection{The Dyson equation}\label{subsec:Dyson}

This subsection is devoted to the first step in the construction of the mapping $\mu$ that we outlined in the introduction. It consists in showing that there exists for every $\eta\in\P(B)$ a holomorphic function $G_\eta: \bH^+(B) \to \bH^-(B)$ uniquely determined among all functions defined on $\bH^+(B)$ and taking values in $\bH^-(B)$ by the Dyson equation \eqref{eq:Dyson_equation}; note that in the sequel, we shall refer to \eqref{eq:opval_semicircular_Cauchy} instead of \eqref{eq:Dyson_equation}. The next theorem, a von Neumann algebra version of which was used also in \cite{AEK2020}, gives the precise statement.

\begin{theorem}\label{thm:opval_semicircular_Cauchy}
Let $\eta: B \to B$ be a positive linear map. Then there exists a unique function $G_\eta: \bH^+(B) \to \bH^-(B)$ solving the \emph{Dyson equation}
\begin{equation}\label{eq:opval_semicircular_Cauchy}
b G_\eta(b) = \1 + \eta(G_\eta(b)) G_\eta(b) \qquad \text{for all $b\in\bH^+(B)$}.
\end{equation}
This function $G_\eta$ is holomorphic and enjoys the property that
\begin{equation}\label{eq:Cauchy_bound-1}
\|G_\eta(b)\| \leq \|\Im(b)^{-1}\| \qquad\text{for all $b\in\bH^+(B)$}.
\end{equation}
\end{theorem}

Essentially, this result is proven in \cite{HRFS2007}, except that the authors have not discussed that the solution of the Dyson equation is holomorphic, because this observation was irrelevant for their purpose. Here, we supplement that argument. To this end, we first recall another result from \cite{HRFS2007} and we take a closer look at its proof.

\begin{theorem}[Proposition 3.2, \cite{HRFS2007}]\label{thm:iteration}
Let $\eta: B \to B$ be a positive linear map. For every fixed $b\in \bH^+(B)$, the equation
\begin{equation}\label{eq:opval_semicircular}
b w = \1 + \eta(w) w
\end{equation}
has a unique solution $w\in \bH^-(B)$; this $w$ is the limit of the sequence $(h_b^{\circ n}(w_0))_{n=1}^\infty$ of iterates of the holomorphic function
$$h_b:\ \bH^-(B) \to \bH^-(B), \quad w \mapsto (b - \eta(w))^{-1}$$
for every choice of an initial point $w_0\in \bH^-(B)$. Moreover, we have that $\|w\| \leq \|\Im(b)^{-1}\|$.
\end{theorem}

Notice that our formulation of Theorem \ref{thm:iteration} differs from the original one in \cite{HRFS2007} as we prefer to perform the iteration in $\bH^-(B)$ instead of the right half-plane of $B$, which would be $\{b\in B \mid \exists \epsilon>0: \Re(b) \geq \epsilon\1\}$, where $\Re(b) := \frac{1}{2}(b+b^\ast)$.

The proof of Theorem \ref{thm:iteration} given in \cite{HRFS2007} crucially relies on the Earle-Hamilton fixed point theorem \cite{EH1970}, an important result about fixed points of holomorphic functions between subsets of complex Banach spaces; we refer the readers who are interested in this topic also to the excellent exposition given in \cite{Harris2003}. For our purpose, it suffices to recap what the authors of \cite{HRFS2007} have proved (in the translation to our setting) in order to apply the Earle-Hamilton fixed point theorem: not only have they shown that $h_b(w) = (b - \eta(w))^{-1}$ defines a holomorphic function $h_b: \bH^-(B) \to \bH^-(B)$ for every $b\in \bH^+(B)$, but they also proved that each $h_b$ is uniformly bounded on $\bH^-(B)$ with $\sup_{w\in \bH^-(B)} \|h_b(w)\| \leq \|\Im(b)^{-1}\|$.
With these insights, we are ready to return to Theorem \ref{thm:opval_semicircular_Cauchy}.

\begin{proof}[Proof of Theorem \ref{thm:opval_semicircular_Cauchy}]
By assigning to $b\in\bH^+(B)$ the unique solution $G_\eta(b) := w \in \bH^-(B)$ of \eqref{eq:opval_semicircular}, we obtain according to Theorem \ref{thm:iteration} a unique function $G_\eta: \bH^+(B) \to \bH^-(B)$ which solves the Dyson equation \eqref{eq:opval_semicircular_Cauchy} and we infer that $G_\eta$ satisfies in addition the bound \eqref{eq:Cauchy_bound-1}. It remains to prove that $G_\eta$ is holomorphic.

To this end, we fix any point $w_0 \in \bH^-(B)$ and use it to define a sequence $(G_n)_{n=1}^\infty$ of functions $G_n: \bH^+(B) \to \bH^-(B)$ by $G_n(b) := h_b^{\circ n}(w_0)$ for all $b \in \bH^+(B)$; inductively, one easily sees that each $G_n$ is G\^ateaux holomorphic. Theorem \ref{thm:iteration} tells us that the sequence $(G_n)_{n=1}^\infty$ converges pointwise on $\bH^+(B)$ to $G_\eta$. Therefore, Lemma \ref{lem:holomorphy_criterion} will imply that $G_\eta$ is holomorphic once we have verified that $(G_n)_{n=1}^\infty$ is locally uniformly bounded.

This can be done as follows. Thanks to the uniform boundedness of $h_b$, we have $\|G_n(b)\| = \|h_b^{\circ n}(w_0)\| \leq \|\Im(b)^{-1}\|$ for all $n\in\N$ and each $b \in \bH^+(B)$. This shows that the sequence $(G_n)_{n=1}^\infty$ is uniformly bounded on $\overline{\bH^+_\gamma(B)}$ for every $\gamma>0$ and hence locally uniformly bounded on $\bH^+(B)$ since we have $\bH^+(B) = \bigcup_{\gamma>0} \overline{\bH^+_\gamma(B)}$.
\end{proof}

Since holomorphy entails Fr\'echet differentiability, we have the Fr\'echet derivative $D G_\eta: \bH^+(B) \to \cL(B)$ of $G_\eta$ at our disposal. By definition, it satisfies for every $b\in \bH^+(B)$
$$\lim_{h \rightarrow 0} \frac{1}{\|h\|} \|G_\eta(b+h) - G_\eta(b) - (DG_\eta)(b) h\| = 0.$$
For later use, we record with the following lemma some strengthening of the latter condition.

\begin{lemma}\label{lem:Taylor_bound_Cauchy}
Let $\eta: B \to B$ be a positive linear map and let $G_\eta: \bH^+(B) \to \bH^-(B)$ be the function uniquely determined by the Dyson equation \eqref{eq:opval_semicircular_Cauchy} according to Theorem \ref{thm:opval_semicircular_Cauchy}. Then the Fr\'echet derivative $D G_\eta: \bH^+(B) \to \cL(B)$ of $G_\eta$ satisfies for each $\gamma>0$
$$\lim_{h \rightarrow 0} \frac{1}{\|h\|} \sup_{b\in \overline{\bH^+_\gamma(B)}} \|G_\eta(b+h) - G_\eta(b) - (DG_\eta)(b) h\| = 0.$$
\end{lemma}

Before giving the proof of Lemma \ref{lem:Taylor_bound_Cauchy}, we first introduce some notation. For $b_0\in \bH^+(B)$ and $\sigma \in (0,1)$, we set $\D(b_0,\sigma) := \{b \in B \mid \|b-b_0\| < \sigma \|\Im(b_0)^{-1}\|^{-1}\}$. Notice that
\begin{equation}\label{eq:D-inclusion}
\D(b_0,\sigma) \subset \bH^+_\gamma(B) \qquad\text{for}\qquad \gamma := (1-\sigma) \|\Im(b_0)^{-1}\|^{-1}.
\end{equation}
Indeed, if any $b\in \D(b_0,\sigma)$ is given, we may check for every state $\phi$ on $B$ that
\begin{multline*}
\phi(\Im(b)) = \phi(\Im(b_0)) + \phi(\Im(b) - \Im(b_0))\\
\geq \|\Im(b_0)^{-1}\|^{-1} - \|b-b_0\| > (1-\sigma) \|\Im(b_0)^{-1}\|^{-1} = \gamma,
\end{multline*}
where we used that $\Im(b_0) \geq \|\Im(b_0)^{-1}\|^{-1} \1$; this shows $b \in \bH^+_\gamma(B)$, as claimed.

\begin{proof}[Proof of Lemma \ref{lem:Taylor_bound_Cauchy}]
Let $\gamma>0$ be given and choose $\sigma\in(0,1)$. Take any $b_0 \in \overline{\bH^+_\gamma(B)}$. Since $\|\Im(b_0)^{-1}\|^{-1} \geq \gamma$, we get from \eqref{eq:D-inclusion} that $\B_B(b_0,\sigma\gamma) \subseteq \D(b_0,\sigma) \subset \bH^+_{(1-\sigma) \gamma}(B)$, and \eqref{eq:Cauchy_bound-1} yields $\sup_{b\in \overline{\B_B(b_0,\sigma\gamma)}} \|G_\eta(b)\| \leq \frac{1}{(1-\sigma) \gamma}$. Using this fact, we deduce from Lemma \ref{lem:Cauchy_Taylor_estimates} \ref{it:Taylor_estimate} that
$$\|G_\eta(b_0+h) - G_\eta(b_0) - (DG_\eta)(b_0) h\| \leq \frac{1}{(1-\sigma) \sigma \gamma^2 (\sigma \gamma - \|h\|)} \|h\|^2$$
for each $h\in B$ which satisfies $\|h\| < \sigma \gamma$. Consequently, as $b_0 \in \overline{\bH^+_\gamma(B)}$ was arbitrary,
$$\frac{1}{\|h\|} \sup_{b_0\in \overline{\bH^+_\gamma(B)}} \|G_\eta(b_0+h) - G_\eta(b_0) - (DG_\eta)(b_0) h\| \leq \frac{1}{(1-\sigma) \sigma \gamma^2 (\sigma \gamma - \|h\|)} \|h\|$$
whenever $h \in B$ satisfies $0 < \|h\| < \sigma\gamma$. This yields the assertion.
\end{proof}

\subsection{Density of states}\label{subsec:density_of_states}

Based on the results about the Dyson equation obtained in the previous section, we finally can construct the map $\mu$ which we announced in the introduction.

We suppose that $(B,\phi)$ is a $C^\ast$-probability space, i.e., a $C^\ast$-algebra $B$ with unit $\1$ which is endowed with a distinguished state $\phi$.

\begin{definition}\label{def:density_of_states}
Let any data pair $\rho = (b_0,\eta) \in B_\sa \times \P(B)$ be given. Consider the holomorphic function $G_\eta: \bH^+(B) \to \bH^-(B)$ associated with $\eta$ by Theorem \ref{thm:opval_semicircular_Cauchy}. Then $\G_\rho(z) := \phi(G_\eta(z\1-b_0))$ for $z\in\C^+$ defines a holomorphic function $\G_\rho: \C^+ \to \C^-$ with the property $\lim_{y \to \infty} iy\G_\rho(iy) = 1$ and is thus, by Theorem \ref{thm:Cauchy_characterization}, the Cauchy transform of a Borel probability measure $\mu_\rho$ on the real line $\R$; more explicitly, we have that
$$\phi(G_\eta(z\1-b_0)) = \int_\R \frac{1}{z-t}\, d\mu_\rho(t) \qquad\text{for all $z\in\C^+$},$$
which determines $\mu_\rho$ uniquely. We call $\mu_\rho$ the \emph{density of states} associated with the data pair $\rho$. In this way, we obtain a mapping
$$\mu:\ B_\sa \times \P(B) \to \Prob(\R),\quad \rho \mapsto \mu_\rho.$$
\end{definition}

Notice that $\lim_{y \to \infty} iy\G_\rho(iy) = 1$ follows easily from \eqref{eq:opval_semicircular_Cauchy} with the help of \eqref{eq:Cauchy_bound-1}.

Definition \ref{def:density_of_states} extends \cite[Definition 2.2]{AEK2020} to the setting of $C^\ast$-probability spaces which are not necessarily tracial and removes any additional assumptions on the positive map.

We emphasize that $\mu_\rho$, despite its name, may very well possess an atomic part and hence can fail to have a density with respect to the Lebesgue measure on $\R$. To ensure that $\mu_\rho$ enjoys regularity properties of this kind, one needs to impose further conditions on $\rho$; see, for instance, \cite{AEK2020,MSY2018}. 

\subsection{Matricial amplifications of the Dyson equation}\label{subsec:Dyson_amplification}

Theorem \ref{thm:opval_semicircular_Cauchy} characterizes $G_\eta$ as the unique function on $\bH^+(B)$ with values in $\bH^-(B)$ which solves the Dyson equation \eqref{eq:opval_semicircular_Cauchy}. In order to further analyze $G_\eta$, it is helpful to take also matricial amplifications $\eta^{(n)}$ of $\eta$ into consideration, especially for $n=2$. For this purpose, we need a better understanding of the structure of $\bH^+(M_2(B))$; this goal is achieved by the next lemma.

\begin{lemma}\label{lem:half-plane}
Let $b_1,b_2 \in \bH^+(B)$ and $w\in B$ be given. Then $b := \begin{bmatrix} b_1 & w\\ 0 & b_2 \end{bmatrix} \in M_2(B)$ belongs to $\bH^+(M_2(B))$ if and only if
\begin{equation}\label{eq:half-plane-condition}
\| \Im(b_2)^{-1/2} w^\ast \Im(b_1)^{-1} w \Im(b_2)^{-1/2} \| < 4.
\end{equation}
Consequently, we have the following implications:
\begin{enumerate}
 \item\label{it:half-plane-1} If $\|w\|^2 < 4 \|\Im(b_1)^{-1}\|^{-1} \|\Im(b_2)^{-1}\|^{-1}$, then $b\in \bH^+(M_2(B))$.
 \item\label{it:half-plane-2} If $b\in \bH^+(M_2(B))$, then $\|w\|^2 < 4 \|\Im(b_1)\| \|\Im(b_2)\|$.
\end{enumerate}
\end{lemma}

The proof of Lemma \ref{lem:half-plane} relies substantially on the following general facts.

\begin{lemma}\label{lem:positive_invertible_matrix}
Let $B$ be a $C^\ast$-algebra with unit $\1$. For each $v\in B$, we have that
\begin{enumerate}
 \item\label{it:positive_matrix} $\begin{bmatrix} \1 & v\\ v^\ast & \1 \end{bmatrix} \in M_2(B)$ is positive if and only if $\|v\| \leq 1$;
 \item\label{it:positive_invertible_matrix} $\begin{bmatrix} \1 & v\\ v^\ast & \1 \end{bmatrix} \in M_2(B)$ is positive and invertible if and only if $\|v\| < 1$.
\end{enumerate}
\end{lemma}

\begin{proof}
The first equivalence \ref{it:positive_matrix} is \cite[Lemma 3.1 (i)]{Paulsen2002}; a direct proof of \ref{it:positive_invertible_matrix} was given in \cite[Section 2.1]{Belinschi2017}. We note that \ref{it:positive_invertible_matrix} can also be deduced from \ref{it:positive_matrix}; indeed, for any $\epsilon\in(0,1)$, it gives that
$\begin{bmatrix} \1 & v\\ v^\ast & \1 \end{bmatrix} \geq \epsilon \1_2,$
or equivalently
$\begin{bmatrix} \1 & \frac{1}{1-\epsilon} v\\ \frac{1}{1-\epsilon}v^\ast & \1 \end{bmatrix} \geq 0,$
if and only if $\|v\| \leq 1-\epsilon$.
\end{proof}

\begin{proof}[Proof of Lemma \ref{lem:half-plane}]
The proof of the first assertion is basically given in the discussion at the end of Section 2.1 in \cite{Belinschi2017}; for the readers' convenience, we include the argument, adapted to our situation. First of all, we notice that $b\in \bH^+(M_2(B))$ if and only if the matrix
$$\Im(b) = \Im\bigg(\begin{bmatrix} b_1 & w\\ 0 & b_2 \end{bmatrix}\bigg) = \begin{bmatrix} \Im(b_1) & \frac{1}{2i}w\\ -\frac{1}{2i}w^\ast & \Im(b_2) \end{bmatrix}$$
is positive and invertible. Since $b_1,b_2 \in \bH^+(B)$ by assumption, both $\Im(b_1)$ and $\Im(b_2)$ are positive and invertible; thus, the operator $\Im(b)$ is positive and invertible if and only if
$$\begin{bmatrix} \Im(b_1)^{-1/2} & 0 \\ 0 & \Im(b_2)^{-1/2} \end{bmatrix} \Im(b) \begin{bmatrix} \Im(b_1)^{-1/2} & 0 \\ 0 & \Im(b_2)^{-1/2} \end{bmatrix} = \begin{bmatrix} \1 & v\\ v^\ast & \1 \end{bmatrix},$$
with $v := \frac{1}{2i} \Im(b_1)^{-1/2} w \Im(b_2)^{-1/2}$, is positive and invertible. With the aid of Lemma \ref{lem:positive_invertible_matrix} \ref{it:positive_invertible_matrix}, we infer that $\Im(b)$ is positive and invertible if and only if
$$\frac{1}{4}\| \Im(b_2)^{-1/2} w^\ast \Im(b_1)^{-1} w \Im(b_2)^{-1/2} \| = \|v^\ast v\| = \|v\|^2 < 1,$$
which is the statement we wanted to prove.

Having established that $b\in \bH^+(M_2(B))$ holds true if and only if condition \eqref{eq:half-plane-condition} is satisfied, we can deduce the additional assertions:

Suppose that $\|w\|^2 < 4 \|\Im(b_1)^{-1}\|^{-1} \|\Im(b_2)^{-1}\|^{-1}$ is satisfied. Since $\Im(b_i)^{-1} \leq \|\Im(b_i)^{-1}\| \1$ for $i=1,2$, we get in this case that
$$\|\Im(b_2)^{-1/2} w^\ast \Im(b_1)^{-1} w \Im(b_2)^{-1/2}\| \leq \|w\|^2 \|\Im(b_1)^{-1}\| \|\Im(b_2)^{-1}\| < 4,$$
Thus, \eqref{eq:half-plane-condition} is satisfied and hence we have $b\in\bH^+(M_2(B))$; this proves \ref{it:half-plane-1}.

In order to prove \ref{it:half-plane-2}, we suppose that $b\in \bH^+(M_2(B))$. First, we notice that $\|w\| \leq \|w \Im(b_2)^{-1/2}\| \| \Im(b_2) \|^{1/2}$ since $\|\Im(b_2)^{1/2}\| = \|\Im(b_2)\|^{1/2}$ by the $C^\ast$-identity. Then, we use $\1 \leq \|\Im(b_1)\| \Im(b_1)^{-1}$ to infer from the latter that
\begin{align*}
\|w\|^2 &\leq \|w \Im(b_2)^{-1/2}\|^2 \| \Im(b_2) \|\\
        &= \|\Im(b_2)^{-1/2} w^\ast w \Im(b_2)^{-1/2}\| \| \Im(b_2) \|\\
				&\leq \|\Im(b_1)\| \|\Im(b_2)\| \| \Im(b_2)^{-1/2} w^\ast \Im(b_1)^{-1} w \Im(b_2)^{-1/2} \|.
\end{align*}
Involving \eqref{eq:half-plane-condition}, which holds since we have $b\in \bH^+(M_2(B))$ by assumption, we conclude that $\|w\|^2 < 4 \|\Im(b_1)\| \|\Im(b_2)\|$, as desired.
\end{proof}

The next proposition provides a strengthening of Theorem \ref{thm:opval_semicircular_Cauchy} for $n$-positive maps $\eta: B \to B$. In such cases, the Dyson equation uniquely determines a noncommutative function of height $n$, in the terminology that we explain in the Appendix \ref{sec:FreeNCFunctionTheory}; see also Proposition \ref{prop:opval_semicircular_Cauchy_matrix_revisited}.

\begin{proposition}\label{prop:opval_semicircular_Cauchy_matrix}
Let $\eta: B \to B$ be an $n$-positive linear map. Consider the noncommutative sets of height $n$ which are given by $\bH^\pm(M_{\leq n}(B)) := \coprod_{k=1}^n \bH^\pm(M_k(B))$. Then there exists a unique noncommutative function $G^{(\leq n)}_\eta: \bH^+(M_{\leq n}(B)) \to \bH^-(M_{\leq n}(B))$ of height $n$ such that, for each $k=1,\dots,n$, the $k$-th level function $G_\eta^{(k)}: \bH^+(M_k(B)) \to \bH^-(M_k(B))$ satisfies
\begin{equation}\label{eq:opval_semicircular_Cauchy_matrix}
b G_\eta^{(k)}(b) = \1_k + \eta^{(k)}(G_\eta^{(k)}(b)) G_\eta^{(k)}(b) \qquad \text{for all $b\in\bH^+(M_k(B))$}.
\end{equation}
In particular, if $n\geq 2$, then for every $b\in \bH^+(B)$ and each $w\in B$ satisfying $\begin{bmatrix} b & w\\ 0 & b \end{bmatrix} \in \bH^+(M_2(B))$, we have that
\begin{equation}\label{eq:Cauchy-matrix-derivative}
G^{(2)}_\eta\bigg(\begin{bmatrix} b & w\\ 0 & b \end{bmatrix} \bigg) = \begin{bmatrix} G_\eta(b) & (D G_\eta)(b) w \\ 0 & G_\eta(b) \end{bmatrix}.
\end{equation}
\end{proposition}

\begin{proof}
For $k=1,\dots,n$, because $\eta^{(k)}: M_k(B) \to M_k(B)$ is positive, Theorem \ref{thm:opval_semicircular_Cauchy} guarantees the existence of a unique function $G_\eta^{(k)}: \bH^+(M_k(B)) \to \bH^-(M_k(B))$ satisfying \eqref{eq:opval_semicircular_Cauchy_matrix}; in fact, we have that $G_\eta^{(k)} = G_{\eta^{(k)}}$. Thus, in order to establish the existence of a unique noncommutative function $G^{(\leq n)}_\eta: \bH^+(M_{\leq n}(B)) \to \bH^-(M_{\leq n}(B))$ of height $n$ with the properties required in the proposition, it suffices to show that the aforementioned functions $G_\eta^{(k)}$ combine to a noncommutative function of height $n$. To this end, we check the following:

Whenever $1 \leq k,l \leq n$ are such that $k+l \leq n$ and if $b_1 \in \bH^+(M_k(B))$ and $b_2 \in \bH^+(M_l(B))$, then both $G_\eta^{(k+l)}(b_1 \oplus b_2)$ and $G_\eta^{(k)}(b_1) \oplus G_\eta^{(l)}(b_2)$ are in $\bH^-(M_{k+l}(B))$ and solve \eqref{eq:opval_semicircular} for $\eta^{(k+l)}\in \P(M_{k+l}(B))$ at $b_1 \oplus b_2 \in \bH^+(M_{k+l}(B))$. By uniqueness, $G_\eta^{(k+l)}(b_1 \oplus b_2) = G_\eta^{(k)}(b_1) \oplus G_\eta^{(l)}(b_2)$.

For $1\leq k \leq n$, if $b\in \bH^+(M_k(B))$ and an invertible matrix $S\in M_k(\C)$ are given such that $S b S^{-1} \in \bH^+(M_k(B))$ holds, then Theorem \ref{thm:iteration} tells us that
$$G^{(k)}_\eta(b) = \lim_{n\to \infty} h^{\circ n}_b(-i \1_k) \qquad\text{and}\qquad G^{(k)}_\eta(S b S^{-1}) = \lim_{n\to \infty} h^{\circ n}_{S b S^{-1}}(-i \1_k).$$
Notice that the functions $h_b$ and $h_{S b S^{-1}}$ are defined like in Theorem \ref{thm:iteration}, but with respect to $\eta^{(k)}$ instead of $\eta$; both of them map $\bH^-(M_k(B))$ into itself. Since for all $w\in \bH^-(M_k(B))$
\begin{multline*}
h_{S b S^{-1}}(w) = \big(S b S^{-1} - \eta^{(k)}(w)\big)^{-1}
                  = S \big(b - S^{-1} \eta^{(k)}(w) S\big)^{-1} S^{-1}\\
									= S \big(b - \eta^{(k)}(S^{-1} w S)\big)^{-1} S^{-1} = S h_b(S^{-1} w S) S^{-1},
\end{multline*}
we see that $h^{\circ n}_{S b S^{-1}} (-i \1_k) = S^{-1} h_b^{\circ n}(-i \1_k) S$ for all $n\in\N$ and so $G_\eta^{(k)}(S b S^{-1}) = S G_\eta^{(k)}(b) S^{-1}$.

The asserted formula \eqref{eq:Cauchy-matrix-derivative} follows immediately from Theorem \ref{theo:matrix-derivative}.
\end{proof}

\subsection{Properties of solutions of the Dyson equation for $2$-positive maps}\label{subsec:Dyson_2-positive}

With the aid of Proposition \ref{prop:opval_semicircular_Cauchy_matrix}, we now derive bounds for the Fr\'echet derivative of $G_\eta$ for $\eta \in \P_2(B)$, which are in the spirit of \cite[Proposition 3.1]{Belinschi2017}.

\begin{lemma}\label{lem:Cauchy_derivatives_bound}
Suppose that $\eta \in \P_2(B)$. Then, for every $b\in \bH^+(B)$,
\begin{equation}\label{eq:Cauchy_derivatives_bound-1}
 \| (D G_\eta)(b) \| = \sup_{w\in B\colon \|w\| \leq 1} \| (D G_\eta)(b) w \| \leq \|\Im(b)^{-1}\|^2,
\end{equation}
and moreover, if $\phi$ is any state on $B$, then
\begin{equation}\label{eq:Cauchy_derivatives_bound-2}
\sup_{w\in B\colon \|w\| \leq 1} \big|\phi( (D G_\eta)(b) w )\big| \leq - \Im(\phi(G_\eta(b))) \|\Im(b)^{-1}\|.
\end{equation}
\end{lemma}

\begin{proof}
We take $b \in \bH^+(B)$ and $w\in B$ such that the condition $\|w\| \leq 2 (1-\epsilon) \|\Im(b)^{-1}\|^{-1}$ is satisfied for some $\epsilon\in (0,1)$. In this situation, Lemma \ref{lem:half-plane} \ref{it:half-plane-1} tells us that
$$\begin{bmatrix} b & w\\ 0 & b \end{bmatrix} \in \bH^+(M_2(B)).$$
Since $\eta$ is supposed to be $2$-positive, we may apply Proposition \ref{prop:opval_semicircular_Cauchy_matrix}; the formula \eqref{eq:Cauchy-matrix-derivative} yields
\begin{equation}\label{eq:Cauchy_derivatives_bound-0}
G_\eta^{(2)}\bigg(\begin{bmatrix} b & w\\ 0 & b \end{bmatrix}\bigg) = \begin{bmatrix} G_\eta(b) & (D G_\eta)(b) w \\ 0 & G_\eta(b) \end{bmatrix}.
\end{equation}
Since $G_\eta^{(2)}$ takes its values in $\bH^-(M_2(B))$, we conclude from the latter with the help of Lemma \ref{lem:half-plane} \ref{it:half-plane-2} that $\|(D G_\eta)(b) w\| < 2 \|\Im(G_\eta(b))\|$. By rescaling and using \eqref{eq:Cauchy_bound-1}, we infer that every $w\in B$ with $\|w\| \leq 1$ satisfies
$$\|(D G_\eta)(b) w\| < \frac{1}{1-\epsilon} \|\Im(G_\eta(b))\| \|\Im(b)^{-1}\| \leq \frac{1}{1-\epsilon} \|\Im(b)^{-1}\|^2.$$
By letting $\epsilon\searrow0$, we obtain \eqref{eq:Cauchy_derivatives_bound-1}, the first of the asserted formulas.

In order to prove \eqref{eq:Cauchy_derivatives_bound-2}, we proceed as follows. First of all, we deduce from \eqref{eq:Cauchy_derivatives_bound-0} that
$$\begin{bmatrix} \Im(G_\eta(b)) & \frac{1}{2i} \big((D G_\eta)(b) w\big) \\ -\frac{1}{2i} \big((D G_\eta)(b) w\big)^\ast & \Im(G_\eta(b)) \end{bmatrix} = \Im\bigg(\begin{bmatrix} G_\eta(b) & (D G_\eta)(b) w \\ 0 & G_\eta(b) \end{bmatrix} \bigg) \leq 0.$$
By applying the amplification $\phi^{(2)}$ of $\phi$ to the latter, we conclude that the scalar matrix
$$\begin{bmatrix} \phi\big(\Im(G_\eta(b))\big) & \frac{1}{2i} \phi\big((D G_\eta)(b) w\big) \\ -\frac{1}{2i} \overline{\phi\big((D G_\eta)(b) w\big)} & \phi\big(\Im(G_\eta(b))\big) \end{bmatrix}$$
is negative semidefinite; hence, we get that $\frac{1}{2} |\phi((D G_\eta)(b) w)| \leq -\phi(\Im(G_\eta(b)))$. By rescaling, we infer from the latter that if $w \in B$ is such that $\|w\| \leq 1$, then
$$\big|\phi\big((D G_\eta)(b) w\big)\big| \leq - \frac{1}{1-\epsilon} \phi(\Im(G_\eta(b))) \|\Im(b)^{-1}\|.$$
By letting $\epsilon\searrow0$, we get \eqref{eq:Cauchy_derivatives_bound-2}, the second of the asserted formulas.
\end{proof}

The bound \eqref{eq:Cauchy_derivatives_bound-1} can be used to derive, for every $\gamma>0$, Lipschitz continuity of $G_\eta$ on $\overline{\bH^+_\gamma(B)}$ with Lipschitz constant $\frac{1}{\gamma^2}$. We prove a more general result; cf. \cite[Proposition 3.1]{Belinschi2017}.

\begin{lemma}\label{lem:Lipschitz_bound}
Suppose that $\eta \in \P_2(B)$. Then, for all $b_0,b_1\in \bH^+(B)$, we have that
\begin{equation}\label{eq:Lipschitz_bound}
\|G_\eta(b_1)-G_\eta(b_0)\| \leq \|\Im(b_0)^{-1}\| \|\Im(b_1)^{-1}\| \|b_1-b_0\|.
\end{equation}
\end{lemma}

\begin{proof}
Fix $b_0,b_1 \in \bH^+(B)$. We consider the function $f: [0,1] \to \bH^-(B)$ which is defined by $f(t) := G_\eta(t b_1 + (1-t) b_0)$ for $t\in [0,1]$. Then $f$ is differentiable as a $B$-valued function with $f'(t) = (D G_\eta)(t b_1 + (1-t) b_0) (b_1 - b_0)$. By using the bound \eqref{eq:Cauchy_derivatives_bound-1} provided by Lemma \ref{lem:Cauchy_derivatives_bound}, we observe that
$$\|f'(t)\| \leq \big\|\big(t \Im(b_1) + (1-t) \Im(b_0)\big)^{-1}\big\|^2 \|b_1-b_0\|.$$
Since $t \Im(b_1) + (1-t) \Im(b_0) \geq (t \|\Im(b_1)^{-1}\|^{-1} + (1-t) \|\Im(b_0)^{-1}\|^{-1})\, \1$, we have
$$\big\|\big(t \Im(b_1) + (1-t) \Im(b_0)\big)^{-1}\big\| \leq \big(t \|\Im(b_1)^{-1}\|^{-1} + (1-t) \|\Im(b_0)^{-1}\|^{-1}\big)^{-1}.$$
By using $\int^1_0 (\gamma_1 t + \gamma_0 (1-t))^{-2}\, dt = (\gamma_0\gamma_1)^{-1}$ for every $\gamma_0,\gamma_1>0$, we finally obtain that
$$\|G_\eta(b_1) - G_\eta(b_0)\| \leq \int^1_0 \|f'(t)\|\, dt \leq \|\Im(b_0)^{-1}\| \|\Im(b_1)^{-1}\| \|b_1-b_0\|,$$
which is the asserted bound \eqref{eq:Lipschitz_bound}.
\end{proof}

Let us point out that an alternative proof of \eqref{eq:Lipschitz_bound} can be given by evaluating $G^{(2)}_\eta$ at the point $\begin{bmatrix} b_0 & r (b_1 - b_0)\\ 0 & b_1 \end{bmatrix}$ for sufficiently small $r>0$ (using the identity \eqref{eq:intertwining} together with the fact that $G^{(\leq 2)}_\eta$ is a noncommutative function of height $2$ by Proposition \ref{prop:opval_semicircular_Cauchy_matrix}) and examining the result with Lemma \ref{lem:half-plane} in analogy to the proof of Lemma \ref{lem:Cauchy_derivatives_bound}; see also Remark \ref{rem:Lipschitz_bound_details}.

From Lemma \ref{lem:Lipschitz_bound}, we can derive Lipschitz continuity of $D G_\eta$ on $\overline{\bH^+_\gamma(B)}$ for every $\gamma>0$.

\begin{lemma}\label{lem:Lipschitz_bound_derivative}
There is a universal constant $c>0$ such that, for all $\eta \in \P_2(B)$ and $\gamma>0$,
\begin{equation}\label{eq:Lipschitz_bound_derivative}
\|(DG_\eta)(b_1)-(DG_\eta)(b_0)\| \leq \frac{c}{\gamma^3}\, \|b_1-b_0\| \qquad\text{for all $b_0,b_1\in \overline{\bH^+_\gamma(B)}$}.
\end{equation}
\end{lemma}

\begin{proof}
Choose $\sigma \in (0,1)$. Like in the proof of Lemma \ref{lem:Taylor_bound_Cauchy}, we infer from \eqref{eq:D-inclusion} that $\B_B(b,\sigma\gamma) \subseteq \D(b,\sigma) \subset \bH^+_{(1-\sigma)\gamma}(B)$ for every $b\in \overline{\bH^+_\gamma(B)}$. Thus, $f(b) := G_\eta(b_1 + b) - G_\eta(b_0 + b)$ yields a well-defined holomorphic function $f: \B_B(0,\sigma\gamma) \to B$ satisfying $(Df)(0) = (DG_\eta)(b_1) - (DG_\eta)(b_0)$ and moreover $M_f(0,\sigma\gamma) \leq \frac{1}{(1-\sigma)^2 \gamma^2} \|b_1 - b_0\|$ according \eqref{eq:Lipschitz_bound}. With the help of the Cauchy estimate stated in Lemma \ref{lem:Cauchy_Taylor_estimates} \ref{it:Cauchy_estimate}, we derive that $\|(DG_\eta)(b_1) - (DG_\eta)(b_0)\| \leq \frac{1}{\sigma (1-\sigma)^2} \frac{1}{\gamma^3} \|b_1 - b_0\|$, from which \eqref{eq:Lipschitz_bound_derivative} follows with $c = \min_{\sigma \in (0,1)} \frac{1}{\sigma (1-\sigma)^2} = \frac{27}{4}$.
\end{proof}

\section{The operator-valued inviscid Burgers equation for paths in $\P_2(B)$}\label{sec:Burgers_equation}

This section is about an operator-valued version of the inviscid Burgers equation for paths in $\P_2(B)$.
In the scalar-valued framework, the fundamental insight that the inviscid Burgers equation constitutes the free probability counterpart of the heat equation is due to Voiculescu; see \cite{Voi1986}, especially the comments following the proof of Theorem 4.3 therein. In \cite[Theorem 8.3]{ABFN2013}, Anshelevich, Belinschi, Fevrier, and Nica generalized the aforementioned result of Voiculescu by an operator-valued version of the inviscid Burgers equation. On the level of formal series, a variant thereof has already appeared in \cite[Proposition 4.5]{Voi1995}.

In the literature on noncommutative probability and on free probability in particular, there are numerous other places where such a dynamical perspective turned out to be extremely successful; we refer the interested readers for instance to \cite{Biane1998,Bauer2004,CK2014,IU2015,AW2016,Schleissinger2017,FHS2020,Jekel2020,JLS2022} and the references listed therein.

With Theorem \ref{thm:Burgers}, we generalize \cite[Theorem 8.3]{ABFN2013} from invertible completely positive maps to arbitrary $2$-positive maps, to the relatively small price that we are limited to initial conditions which are themselves solutions of some Dyson equation; this, however, is entirely sufficient for the intended application of Theorem \ref{thm:Burgers} in the proof of Theorem \ref{thm:Levy_Hoelder}.

Theorem \ref{thm:Burgers} builds upon Theorem \ref{thm:continuity}, the proof of which will also be given in this section. Both proofs will make use of the following fact, which we borrow from \cite{HMS2021}; because it is stated there in slightly different form and under the much more restrictive assumption of complete positivity instead of $2$-positivity, we include here also the proof.

\begin{proposition}\label{prop:opval_semicircular_approximation}
Let $B$ be a unital $C^\ast$-algebra and let $\eta \in \P_2(B)$ be given. Define holomorphic functions $\Psi_\eta: \bH^-(B) \to B$ by $\Psi_\eta(w) := w^{-1} + \eta(w)$ and $\Delta_{b,\eta}: \bH^-(B) \to B$ by $\Delta_{b,\eta}(w) := b - \Psi_\eta(w)$ for some fixed $b\in \bH^+(B)$. Furthermore, we set
\begin{IEEEeqnarray*}{lCll}
\U_{b,\eta} &:= &\{w \in \bH^-(B) \mid \|\Delta_{b,\eta}(w)\| < \|\Im(b)^{-1}\|^{-1}\} \qquad &\text{and}\\
\U^\sigma_{b,\eta} &:= &\{w \in \bH^-(B) \mid \|\Delta_{b,\eta}(w)\| \leq \sigma \|\Im(b)^{-1}\|^{-1}\} \qquad &\text{for $\sigma \in (0,1)$}.
\end{IEEEeqnarray*}
Then the following statements hold true:
\begin{enumerate}
 \item\label{it:Cauchy_local-inverse} $\Psi_\eta$ restricts to a holomorphic function $\Psi_\eta: \U_{b,\eta} \to \bH^+(B)$ which satisfies $G_\eta(\Psi_\eta(w)) = w$ for all $w \in \U_{b,\eta}$. Further, for every $\sigma\in(0,1)$, we have that $$\Psi_\eta(\U^\sigma_{b,\eta}) \subseteq \overline{\D(b,\sigma)} \subset \overline{\bH^+_\gamma(B)} \qquad\text{for}\qquad \gamma := (1-\sigma) \|\Im(b)^{-1}\|^{-1}.$$
 \item\label{it:approximate_solution-proximity} Put $w_0 := G_\eta(b) \in\bH^-(B)$. Then, for every $\sigma\in(0,1)$ and each $w \in \U^\sigma_{b,\eta}$, we have
\begin{equation}\label{eq:approximate_solution-proximity}
\|w - w_0\| \leq \frac{1}{1-\sigma} \|\Im(b)^{-1}\|^2 \|\Delta_{b,\eta}(w)\| \leq \frac{\sigma}{1-\sigma} \|\Im(b)^{-1}\|.
\end{equation}
\end{enumerate}
\end{proposition}

Note that $\U_{b,\eta}$ is an open set containing $w_0$, which is the unique solution of the equation \eqref{eq:opval_semicircular} in $\bH^-(B)$; indeed, we have $\Delta_{b,\eta}(w_0)=0$ so that in fact $w_0 \in \U^\sigma_{b,\eta}$ for every $\sigma \in (0,1)$. Accordingly, we regard points in $\U_{b,\eta}^\sigma$ as approximate solutions of \eqref{eq:opval_semicircular}, which by \eqref{eq:approximate_solution-proximity} are getting closer to the accurate solution $w_0$ when $\sigma$ tends to $0$.

\begin{proof}[Proof of Proposition \ref{prop:opval_semicircular_approximation}]
(i) It is obvious that $\Psi_\eta$ and hence $\Delta_{b,\eta}$ define $B$-valued holomorphic functions on $\bH^-(B)$. Take any $\sigma\in(0,1)$. Notice that $\|\Psi_\eta(w) - b\| = \|\Delta_{b,\eta}(w)\| \leq \sigma \|\Im(b)^{-1}\|^{-1}$ and hence $\Psi_\eta(w) \in \overline{\D(b,\sigma)}$ for $w \in\U^\sigma_{b,\eta}$; this proves that $\Psi_\eta(\U^\sigma_{b,\eta}) \subseteq \overline{\D(b,\sigma)} \subset \overline{\bH^+_\gamma(B)}$ for $\gamma = (1-\sigma) \|\Im(b)^{-1}\|^{-1}$, where the second inclusion comes from \eqref{eq:D-inclusion}. Because $\U_{b,\eta} = \bigcup_{\sigma\in(0,1)} \U^\sigma_{b,\eta}$, it follows that the restriction of $\Psi_\eta$ to $\U_{b,\eta}$ takes its values in $\bH^+(B)$.

Next, we verify that $G_\eta(\Psi_\eta(w)) = w$ for all $w \in \U_{b,\eta}$. To this end, we take any $w \in \U_{b,\eta}$. First, we observe that $\Psi_\eta(w) w = \1 + \eta(w) w$. This tells us that $w$ is the unique solution of \eqref{eq:opval_semicircular} in $\bH^-(B)$ at the point $\Psi_\eta(w) \in \bH^+(B)$, which implies $G_\eta(\Psi_\eta(w)) = w$, as we wished to show.

(ii) Take any $w\in \U^\sigma_{b,\eta}$. From \ref{it:Cauchy_local-inverse}, we know $\Psi_\eta(w) \in \overline{\bH^+_\gamma(B)}$ for $\gamma = (1-\sigma) \|\Im(b)^{-1}\|^{-1}$; consequently, $\|\Im(\Psi_\eta(w))^{-1}\| \leq \frac{1}{1-\sigma} \|\Im(b)^{-1}\|$.
With the help of the bound \eqref{eq:Lipschitz_bound}, we get
\begin{multline*}
\|w - w_0\| = \|G_\eta(\Psi_\eta(w))-G_\eta(b)\|
            \leq \|\Im(\Psi_\eta(w))^{-1}\| \|\Im(b)^{-1}\| \|\Delta_{b,\eta}(w)\|\\
						\leq \frac{1}{1-\sigma} \|\Im(b)^{-1}\|^2 \|\Delta_{b,\eta}(w)\|
						\leq \frac{\sigma}{1-\sigma} \|\Im(b)^{-1}\|
\end{multline*}
for every $w\in \U^\sigma_{b,\eta}$, which are the bounds asserted in \eqref{eq:approximate_solution-proximity}.
\end{proof}

We continue these investigations with the following two propositions which allow a local comparison of solutions $G_{\eta_0}$ and $G_{\eta_1}$ of the Dyson equation \eqref{eq:opval_semicircular_Cauchy} for $\eta_0,\eta_1 \in \P_2(B)$.
With the first one, which builds on Proposition \ref{prop:opval_semicircular_approximation} \ref{it:Cauchy_local-inverse}, we show that $G_{\eta_0}$ and $G_{\eta_1}$ are locally subordinated (to each other) whenever $\eta_0$ and $\eta_1$ in $\P_2(B)$ are sufficiently close and we control in terms of $\|\eta_1-\eta_0\|$ the deviation of the subordination function from the identity.

\begin{proposition}\label{prop:local_subordination}
Fix $\gamma > 0$ and $0 < \sigma' < \sigma  < 1$. Suppose that $\eta_0,\eta_1 \in \P_2(B)$ are such that the condition $\|\eta_1-\eta_0\| \leq (1-\sigma')(\sigma-\sigma') \gamma^2$ is satisfied. Then, for each $b_0 \in \overline{\bH^+_\gamma(B)}$, we have the inclusion $G_{\eta_1}(\overline{\D(b_0,\sigma')}) \subseteq \U^\sigma_{b_0,\eta_0}$, so that we obtain a well-defined function by
$$\omega^{b_0}_{\eta_0 \to \eta_1}:\ \overline{\D(b_0,\sigma')} \to \overline{\D(b_0,\sigma)},\quad b \mapsto \Psi_{\eta_0}(G_{\eta_1}(b)),$$
which is holomorphic on $\D(b_0,\sigma')$ and furthermore has the property that $G_{\eta_0}(\omega^{b_0}_{\eta_0 \to \eta_1}(b)) = G_{\eta_1}(b)$ holds for all $b\in \overline{\D(b_0,\sigma')}$. Moreover, we have that
\begin{align}
\label{eq:local_subordination-1}\sup_{b_0 \in \overline{\bH^+_\gamma(B)}}\ \sup_{b\in \overline{\D(b_0,\sigma')}} \|\omega^{b_0}_{\eta_0 \to \eta_1}(b) - b\| &\leq \frac{1}{(1-\sigma')\gamma} \|\eta_1 - \eta_0\| \qquad\text{and}\\
\label{eq:local_subordination-2}\sup_{b_0 \in \overline{\bH^+_\gamma(B)}} \|(D\omega^{b_0}_{\eta_0 \to \eta_1})(b_0) - \id_B\| &\leq \frac{1}{\sigma'(1-\sigma')\gamma^2} \|\eta_1 - \eta_0\|.
\end{align}
\end{proposition}

Before we proceed to the proof of Proposition \ref{prop:local_subordination}, we collect in the following remark some useful identities to which we will come back in the sequel.

\begin{remark}\label{rem:Delta}
Let $\eta_0,\eta_1 \in \P_2(B)$, $\gamma>0$, and $\sigma\in(0,1)$ be given. For $b_0 \in \overline{\bH^+_\gamma(B)}$ and $b \in \overline{\D(b_0,\sigma)}$, we may compute using \eqref{eq:opval_semicircular_Cauchy} for $G_{\eta_1}$ that
\begin{equation}\label{eq:Delta-computation}
\Delta_{b_0,\eta_0}(G_{\eta_1}(b)) = b_0 - G_{\eta_1}(b)^{-1} - \eta_0(G_{\eta_1}(b)) = b_0 - b + (\eta_1 - \eta_0)(G_{\eta_1}(b))
\end{equation}
and in particular $\Delta_{b_0,\eta_0}(G_{\eta_1}(b_0)) = (\eta_1 - \eta_0)(G_{\eta_1}(b_0))$. From \eqref{eq:D-inclusion}, we know that $\overline{\D(b_0,\sigma)} \subset \overline{\bH^+_{(1-\sigma)\gamma}(B)}$; thanks to \eqref{eq:Cauchy_bound-1}, we thus have that $\|G_{\eta_1}(b)\| \leq \frac{1}{(1-\sigma)\gamma}$. It follows from \eqref{eq:Delta-computation} that
\begin{equation}\label{eq:Delta-bound}
\|\Delta_{b_0,\eta_0}(G_{\eta_1}(b))\| \leq \|b-b_0\| + \|\eta_1 - \eta_0\| \|G_{\eta_1}(b)\| \leq \sigma \|\Im(b_0)^{-1}\|^{-1} + \frac{1}{(1-\sigma)\gamma} \|\eta_1 - \eta_0\|
\end{equation}
and in particular $\|\Delta_{b_0,\eta_0}(G_{\eta_1}(b_0))\| \leq \frac{1}{\gamma} \|\eta_1-\eta_0\|$.
\end{remark}

\begin{proof}[Proof of Proposition \ref{prop:local_subordination}]
Suppose that $\eta_0,\eta_1 \in \P_2(B)$ satisfy the condition $\|\eta_1-\eta_0\| \leq (1-\sigma')(\sigma-\sigma') \gamma^2$. We take $b_0 \in \overline{\bH^+_\gamma(B)}$ and any $b \in \overline{\D(b_0,\sigma')}$. Using that $\gamma \leq \|\Im(b_0)^{-1}\|^{-1}$, we obtain with the aid of the bound \eqref{eq:Delta-bound} from Remark \ref{rem:Delta} that
$$\|\Delta_{b_0,\eta_0}(G_{\eta_1}(b))\| \leq \sigma' \|\Im(b_0)^{-1}\|^{-1} + (\sigma-\sigma') \gamma \leq \sigma \|\Im(b_0)^{-1}\|^{-1},$$
which means that $G_{\eta_1}(b) \in \U_{b_0,\eta_0}^\sigma$. This verifies the asserted inclusion $G_{\eta_1}(\overline{\D(b_0,\sigma')}) \subseteq \U^\sigma_{b_0,\eta_0}$.

From Proposition \ref{prop:opval_semicircular_approximation} \ref{it:Cauchy_local-inverse}, we learn that $\Psi_{\eta_0}(\U^\sigma_{b_0,\eta_0}) \subseteq \overline{\D(b_0,\sigma)}$ and $G_{\eta_0}(\Psi_{\eta_0}(w)) = w$ for all $w\in \U^\sigma_{b_0,\eta_0}$; hence, $\omega^{b_0}_{\eta_0 \to \eta_1}: \overline{\D(b_0,\sigma')} \to \overline{\D(b_0,\sigma)}$ is well-defined and satisfies $G_{\eta_0}(\omega^{b_0}_{\eta_0 \to \eta_1}(b)) = G_{\eta_1}(b)$ for all $b\in \overline{\D(b_0,\sigma')}$. The holomorphy of $\Psi_{\eta_0}$ and $G_{\eta_1}$ established in Proposition \ref{prop:opval_semicircular_approximation} respectively in Theorem \ref{thm:opval_semicircular_Cauchy} ensures that $\omega^{b_0}_{\eta_0 \to \eta_1}$ is holomorphic on $\D(b_0,\sigma')$, as asserted.

To prove \eqref{eq:local_subordination-1}, we take $b_0 \in \overline{\bH^+_\gamma(B)}$ and $b\in \overline{\D(b_0,\sigma')}$ and check with the help of Remark \ref{rem:Delta} that $b - \omega^{b_0}_{\eta_0 \to \eta_1}(b) = \Delta_{b,\eta_0}(G_{\eta_1}(b)) = (\eta_1 - \eta_0)(G_{\eta_1}(b))$ and in turn $\|\omega^{b_0}_{\eta_0 \to \eta_1}(b) - b\| \leq \frac{1}{(1-\sigma')\gamma} \|\eta_1 - \eta_0\|$. By taking suprema over $b\in \overline{\D(b_0,\sigma')}$ and $b_0 \in \overline{\bH^+_\gamma(B)}$, we get \eqref{eq:local_subordination-1}.

To prove \eqref{eq:local_subordination-2}, we proceed as follows. Take an arbitrary $b_0 \in \overline{\bH^+_\gamma(B)}$. Because $\B_B(b_0,\sigma'\gamma) \subseteq \D(b_0,\sigma')$, we get by $f: \B_B(b_0,\sigma'\gamma) \to B, b \mapsto \omega^{b_0}_{\eta_0 \to \eta_1}(b) - b$ a well-defined holomorphic function. From Lemma \ref{lem:Cauchy_Taylor_estimates} \ref{it:Cauchy_estimate}, we infer that
$$\|(D \omega^{b_0}_{\eta_0 \to \eta_1})(b_0) - \id_B\| = \|(Df)(b_0)\| \leq \frac{1}{\sigma'\gamma} M_f(b_0,\sigma'\gamma) \leq \frac{1}{\sigma'\gamma} \sup_{b\in \overline{\D(b_0,\sigma')}} \|\omega^{b_0}_{\eta_0 \to \eta_1}(b) - b\|.$$
By passing to the supremum over all $b_0 \in \overline{\bH^+_\gamma(B)}$ and employing \eqref{eq:local_subordination-1}, we arrive at \eqref{eq:local_subordination-2}.
\end{proof}

The second proposition exploits Proposition \ref{prop:opval_semicircular_approximation} \ref{it:approximate_solution-proximity}.

\begin{proposition}\label{prop:local_comparison}
Fix $\gamma > 0$ and $\sigma_0 \in (0,1)$. Suppose that $\eta_0,\eta_1 \in \P_2(B)$ are such that the condition $\|\eta_1-\eta_0\| \leq \sigma_0 \gamma^2$ is satisfied. Then, we have
\begin{equation}\label{eq:local_comparison}
\sup_{b \in \overline{\bH^+_\gamma(B)}} \|G_{\eta_1}(b) - G_{\eta_0}(b)\| \leq \frac{1}{(1-\sigma_0)\gamma^3} \|\eta_1-\eta_0\|
\end{equation}
and moreover, with the constant $c>0$ which appeared in Lemma \ref{lem:Lipschitz_bound_derivative},
\begin{equation}\label{eq:local_comparison_derivative}
\sup_{b \in \overline{\bH^+_\gamma(B)}} \|(D G_{\eta_1})(b) - (D G_{\eta_0})(b)\| \leq \frac{1}{\gamma^4} \Bigg[\min_{\substack{0 < \sigma' < \sigma < 1\colon\\ \sigma_0 = (1-\sigma')(\sigma-\sigma')}} \frac{1 - \sigma + c \sigma'}{\sigma'(1-\sigma') (1-\sigma)^3}\Bigg] \|\eta_1 - \eta_0\|.
\end{equation}
\end{proposition}

\begin{proof}
Let us consider arbitrary $0 < \sigma' < \sigma < 1$ such that $\sigma_0 = (1-\sigma')(\sigma-\sigma')$.

Take any $b_0 \in \overline{\bH^+_\gamma(B)}$. Thanks to Proposition \ref{prop:local_subordination}, we thus have that $G_{\eta_1}(\overline{\D(b_0,\sigma')}) \subseteq \U^\sigma_{b_0,\eta_0}$ and in particular $G_{\eta_1}(b_0) \in \U^\sigma_{b_0,\eta_0}$. Hence, we may apply Proposition \ref{prop:opval_semicircular_approximation} \ref{it:approximate_solution-proximity}, which gives us $\|G_{\eta_1}(b_0) - G_{\eta_0}(b_0)\| \leq \frac{1}{(1-\sigma)\gamma^2} \|\Delta_{b_0,\eta_0}(G_{\eta_1}(b_0))\|$. 
Further, due to Remark \ref{rem:Delta}, we have that $\|\Delta_{b_0,\eta_0}(G_{\eta_1}(b_0))\| \leq \frac{1}{\gamma} \|\eta_1-\eta_0\|$. 
In combination, we get $\|G_{\eta_1}(b_0) - G_{\eta_0}(b_0)\| \leq \frac{1}{(1-\sigma)\gamma^3} \|\eta_1-\eta_0\|$, which yields \eqref{eq:local_comparison} by taking the supremum over all $b_0 \in \overline{\bH^+_\gamma(B)}$ and letting $\sigma' \searrow 0$.

Take again any $b_0 \in \overline{\bH^+_\gamma(B)}$. From Proposition \ref{prop:local_subordination}, we know that $\omega^{b_0}_{\eta_0 \to \eta_1}: \D(b_0,\sigma') \to \overline{\D(b_0,\sigma)}$ is holomorphic and satisfies $G_{\eta_0}(\omega^{b_0}_{\eta_0 \to \eta_1}(b)) = G_{\eta_1}(b)$ for all $b\in \D(b_0,\sigma')$. This identity allows us to compute the Fr\'echet derivative of $G_{\eta_1}$ at the point $b_0$ as $(D G_{\eta_1})(b_0) = (D G_{\eta_0})(\omega^{b_0}_{\eta_0 \to \eta_1}(b_0)) (D \omega^{b_0}_{\eta_0 \to \eta_1})(b_0)$. Therefore, we can estimate
\begin{multline}\label{eq:bounding_derivatives}
\big\|(D G_{\eta_1})(b_0) - (D G_{\eta_0})(b_0)\big\|
\leq \big\|(D G_{\eta_0})(\omega^{b_0}_{\eta_0 \to \eta_1}(b_0)) \big((D \omega^{b_0}_{\eta_0 \to \eta_1})(b_0) - \id_B\big)\big\|\\ + \big\|(D G_{\eta_0})(\omega^{b_0}_{\eta_0 \to \eta_1}(b_0)) - (D G_{\eta_0})(b_0)\big\|.
\end{multline}
Keeping in mind that both $b_0$ and, by \eqref{eq:D-inclusion}, $\omega^{b_0}_{\eta_0 \to \eta_1}(b_0)$ belong to $\overline{\bH^+_{(1-\sigma)\gamma}(B)}$, we obtain
$$\big\|(D G_{\eta_0})(\omega^{b_0}_{\eta_0 \to \eta_1}(b_0)) \big((D \omega^{b_0}_{\eta_0 \to \eta_1})(b_0) - \id_B\big)\big\| \leq \frac{1}{(1-\sigma)^2 \gamma^2} \|(D \omega^{b_0}_{\eta_0 \to \eta_1})(b_0) - \id_B\|$$
from the bound \eqref{eq:Cauchy_derivatives_bound-1} provided by Lemma \ref{lem:Cauchy_derivatives_bound}, and from Lemma \ref{lem:Lipschitz_bound_derivative} we deduce that
$$\|(D G_{\eta_0})(\omega^{b_0}_{\eta_0 \to \eta_1}(b_0)) - (D G_{\eta_0})(b_0)\| \leq \frac{c}{(1-\sigma)^3 \gamma^3} \|\omega^{b_0}_{\eta_0 \to \eta_1}(b_0) - b_0\|.$$
By involving the bounds \eqref{eq:local_subordination-2} and \eqref{eq:local_subordination-1}, respectively, the latter estimates lead us to
\begin{multline*}
\big\|(D G_{\eta_0})(\omega^{b_0}_{\eta_0 \to \eta_1}(b_0)) \big((D \omega^{b_0}_{\eta_0 \to \eta_1})(b_0) - \id_B\big)\big\| \leq \frac{1}{\sigma' (1-\sigma') (1-\sigma)^2 \gamma^4} \|\eta_1 - \eta_0\| \qquad\text{and}\\
\|(D G_{\eta_0})(\omega^{b_0}_{\eta_0 \to \eta_1}(b_0)) - (D G_{\eta_0})(b_0)\| \leq \frac{c}{(1-\sigma')(1-\sigma)^3 \gamma^4} \|\eta_1 - \eta_0\|,
\end{multline*}
respectively. By combining these insights with \eqref{eq:bounding_derivatives}, passing to the supremum over $b_0 \in \overline{\bH^+_\gamma(B)}$, and optimizing over $0 < \sigma' < \sigma <1$ with $\sigma_0 = (1-\sigma')(\sigma-\sigma')$, we obtain \eqref{eq:local_comparison_derivative}.
\end{proof}

With the help of Proposition \ref{prop:local_comparison}, we can prove Theorem \ref{thm:continuity}.

\begin{proof}[Proof of Theorem \ref{thm:continuity}]
First, we notice that the assertions in Item \ref{it:continuity-0} and Item \ref{it:continuity-1} concerning the continuity of $G: \P_2(B) \times \bH^+(B) \to \bH^-(B)$ and $DG: \P_2(B) \times \bH^+(B) \to \cL(B)$ immediately follow once the equicontinuity in the respective case is established. The general mechanism behind (albeit not in its most general version but tailored to our situation) is the following: if $f: \P_2(B) \times \bH^+(B) \to F$, for some Banach space $(F,\|\cdot\|)$, is such that
\begin{itemize}
 \item for all $\gamma>0$, the family $f(\cdot,b): \P_2(B) \to F$ with $b \in \overline{\bH^+_\gamma(B)}$ is equicontinuous, and
 \item for every fixed $\eta \in \P_2(B)$, the function $f(\eta,\cdot): \bH^+(B) \to F$ is continuous,
\end{itemize}
then $f$ must be continuous. To see this, we take any sequence $((\eta_n,b_n))_{n=1}^\infty$ in $\P_2(B) \times \bH^+(B)$ converging to $(\eta,b) \in \P_2(B) \times \bH^+(B)$. Note that $(b_n)_{n=1}^\infty$ lies in $\overline{\bH^+_\gamma(B)}$ for some $\gamma>0$; thus,
$$\|f(\eta_n,b_n) - f(\eta,b)\| \leq \sup_{b' \in \overline{\bH^+_\gamma(B)}} \|f(\eta_n,b') - f(\eta,b')\| + \|f(\eta,b_n) - f(\eta,b)\|$$
for all $n\in\N$. With the help of the latter estimate, we easily infer from the assumptions on the separate continuity of $f$ that $f(\eta_n,b_n) \to f(\eta,b)$ as $n \to \infty$. Thus, $f$ is continuous.

The statements on the equicontinuity follow immediately from Proposition \ref{prop:local_comparison}.
\end{proof}

With these preparations done, we may get down to prove Theorem \ref{thm:Burgers}.

\begin{proof}[Proof of Theorem \ref{thm:Burgers}]
First of all, we notice that the functions $G: [0,T] \times \bH^+(B) \to \bH^-(B)$ and $DG: [0,T] \times \bH^+(B) \to \cL(B)$ are continuous thanks to Theorem \ref{thm:continuity} as $\eta: [0,T] \to \P_2(B)$ is continuous. Therefore and because $\dot{\eta}: [0,T] \to \P_2(B)$ is continuous, the continuity of $\dot{G}: [0,T] \times \bH^+(B) \to B$ follows once the Burgers equation \eqref{eq:Burgers} is shown.
Moreover, once \eqref{eq:differentiability} is established, Lemma \ref{lem:holomorphy_criterion} yields that the function $\dot{G}_t: \bH^+(B) \to B$ is holomorphic for every fixed $t\in[0,T]$.
For the proof of Theorem \ref{thm:Burgers} it thus suffices to prove \eqref{eq:differentiability} and \eqref{eq:Burgers}.

We fix $\gamma > 0$ and $t_0\in[0,T]$. Let $\epsilon > 0$ be given. By Lemma \ref{lem:Taylor_bound_Cauchy}, we find $r>0$ such that
\begin{equation}\label{eq:Frechet_condition}
\Big\| G_{t_0}(b) - G_{t_0}(b_0) - (D G_{t_0})(b_0)\big(b-b_0\big) \Big\| \leq \frac{\gamma \epsilon}{3 (1 + \|\dot{\eta}_{t_0}\|)}\, \|b-b_0\|
\end{equation}
whenever $b_0 \in \overline{\bH^+_\gamma(B)}$ and $b\in \B_B(b_0,r) \cap \bH^+(B)$. Next, we choose some $\sigma\in(0,1)$ such that
\begin{equation}\label{eq:sigma-condition}
\frac{\sigma}{(1-\sigma)\gamma^3} \|\dot{\eta}_{t_0}\| < \frac{\epsilon}{3} \qquad\text{and}\qquad \sigma \gamma < r.
\end{equation}
By continuity and differentiability of $\eta$ at $t_0$, there exists $\delta>0$ such that the conditions
\begin{multline}\label{eq:differentiability-condition}
\|\eta_t - \eta_{t_0}\| \leq \sigma \gamma^2 \quad\text{for all $t\in I$} \qquad\text{and}\\
\Big\| \frac{1}{t-t_0}(\eta_t - \eta_{t_0}) - \dot{\eta}_{t_0}\Big\| < \min\Big\{1,\frac{\gamma^3\epsilon}{3}\Big\} \quad\text{for all $t\in I \setminus \{t_0\}$}
\end{multline}
are satisfied with $I := [0,T] \cap (t_0-\delta,t_0+\delta)$. Let us consider an arbitrary $b_0 \in \overline{\bH^+_\gamma(B)}$. Thanks to the first condition in \eqref{eq:differentiability-condition}, we may proceed like in the proof of Proposition \ref{prop:local_subordination} and deduce from Remark \ref{rem:Delta} that $G_t(b_0) \in \U_{b_0,\eta_{t_0}}^\sigma$ for all $t\in I$. Using the bound \eqref{eq:Cauchy_derivatives_bound-1} from Lemma \ref{lem:Cauchy_derivatives_bound} and Proposition \ref{prop:opval_semicircular_approximation} \ref{it:approximate_solution-proximity}, we infer from the first condition in \eqref{eq:sigma-condition} that
\begin{equation}\label{eq:term-1}
\|(D G_{t_0})(b_0)\| \|\dot{\eta}_{t_0}\| \|G_t(b_0) - G_{t_0}(b_0)\| < \frac{\epsilon}{3} \qquad\text{for all $t\in I$}.
\end{equation}
By involving Proposition \ref{prop:local_subordination}, we see that $b_{0,t} := \omega^{b_0}_{\eta_{t_0} \to \eta_t}(b_0)$ lies in $\bH^+(B)$ and satisfies $G_t(b_0) = G_{t_0}(b_{0,t})$ for all $t \in I$. From Remark \ref{rem:Delta} and the first condition in \eqref{eq:differentiability-condition}, we get 
\begin{equation}\label{eq:wandering_b}
b_0 - b_{0,t} = \Delta_{b_0,\eta_{t_0}}(G_t(b_0)) = (\eta_t-\eta_{t_0})(G_t(b_0))
\end{equation}
and in turn $\|b_{0,t} - b_0\| \leq \sigma \gamma$ for all $t\in I$; by the second condition imposed in \eqref{eq:sigma-condition}, this ensures that $b_{0,t} \in \B_B(b_0,r) \cap \bH^+(B)$ for all $t\in I$. Thus, we can apply \eqref{eq:Frechet_condition}, which gives us
\begin{equation}\label{eq:term-2}
\frac{1}{|t-t_0|} \Big\| G_t(b_0) - G_{t_0}(b_0) - (D G_{t_0})(b_0)\big(b_{0,t}-b_0\big) \Big\| \leq \frac{\epsilon}{3} \qquad\text{for all $t\in I \setminus \{t_0\}$},
\end{equation}
where we have used that \eqref{eq:wandering_b} and the second condition in \eqref{eq:differentiability-condition} imply
$$\frac{1}{|t-t_0|} \|b_{0,t}-b_0\| \leq \frac{1}{\gamma} \bigg(\Big\| \frac{1}{t-t_0}(\eta_t - \eta_{t_0}) - \dot{\eta}_{t_0}\Big\| + \|\dot{\eta}_{t_0}\|\bigg) \leq \frac{1+\|\dot{\eta}_{t_0}\|}{\gamma}.$$
From \eqref{eq:wandering_b} and the second condition in \eqref{eq:differentiability-condition}, we further obtain by using \eqref{eq:Cauchy_derivatives_bound-1} from Lemma \ref{lem:Cauchy_derivatives_bound} as well as one more time the bound $\|G_t(b_0)\| \leq \frac{1}{\gamma}$ deduced from \eqref{eq:Cauchy_bound-1} that
\begin{equation}\label{eq:term-3}
\|(D G_{t_0})(b_0)\| \Big\|\frac{1}{t-t_0}(b_0-b_{0,t}) - \dot{\eta}_{t_0}(G_t(b_0))\Big\| < \frac{\epsilon}{3} \qquad\text{for all $t\in I \setminus \{t_0\}$}.
\end{equation}
Putting together the bounds \eqref{eq:term-1}, \eqref{eq:term-2}, and \eqref{eq:term-3} which we just established, we obtain
\begin{multline*}
\Big\|\frac{1}{t-t_0}\big(G_t(b_0) - G_{t_0}(b_0)\big) + (D G_{t_0})(b_0)\big(\dot{\eta}_{t_0}(G_{t_0}(b_0))\big) \Big\|\\
\leq \frac{1}{|t-t_0|} \Big\| G_t(b_0) - G_{t_0}(b_0) - (D G_{t_0})(b_0)\big(b_{0,t} - b_0\big) \Big\|\\
+ \|(D G_{t_0})(b_0)\| \Big\|\frac{1}{t-t_0}(b_0 - b_{0,t}) - \dot{\eta}_{t_0}(G_t(b_0))\Big\|
+ \|(D G_{t_0})(b_0)\| \|\dot{\eta}_{t_0}\| \|G_t(b_0) - G_{t_0}(b_0)\|
 < \epsilon
\end{multline*}
for all $t\in I \setminus \{t_0\}$. Because $b_0\in \overline{\bH^+_\gamma(B)}$ was arbitrary, it follows that
$$\sup_{b_0\in \overline{\bH^+_\gamma(B)}} \Big\|\frac{1}{t-t_0}\big(G_t(b_0) - G_{t_0}(b_0)\big) + (D G_{t_0})(b_0)\big(\dot{\eta}_{t_0}(G_{t_0}(b_0))\big) \Big\| \leq \epsilon.$$
Thus, in summary, we see that
$$\lim_{t\to t_0} \sup_{b_0\in \overline{\bH^+_\gamma(B)}} \Big\|\frac{1}{t-t_0}\big(G_t(b_0) - G_{t_0}(b_0)\big) + (D G_{t_0})(b_0)\big(\dot{\eta}_{t_0}(G_{t_0}(b_0))\big) \Big\| = 0.$$
This shows that the family $G_{\boldsymbol{\cdot}}(b_0): [0,T] \to \bH^-(B)$ for $b_0\in \overline{\bH^+_\gamma(B)}$ is equidifferentiable at $t_0$ with $\dot{G}_{t_0}(b_0) = - (D G_{t_0})(b_0)(\dot{\eta}_{t_0}(G_{t_0}(b_0)))$, which establishes at once \eqref{eq:differentiability} and \eqref{eq:Burgers}.
\end{proof}

\section{H\"older continuity of the density of states}\label{sec:Levy_Hoelder}

In this section, we give the proof of Theorem \ref{thm:Levy_Hoelder}. For each of the two H\"older bounds stated therein, the proof follows the same strategy based on the next lemma.

\begin{lemma}\label{lem:Levy_strategy}
Let $(\rho_t)_{t\in[0,T]}$ be a family of data pairs in $B_\sa \times \P(B)$. Suppose that the function $\G: [0,T] \times \C^+ \to \C^-, (t,z) \mapsto \G_t(z)$ defined by $\G_t(z) := \G_{\rho_t}(z)$ for $(t,z) \in [0,T] \times \C^+$ satisfies the following conditions:
\begin{enumerate}
 \item\label{it:Levy_strategy-continuity} The function $\G: [0,T] \times \C^+ \to \C^-$ is continuous with respect to the product topology on $[0,T] \times \C^+$.
 \item\label{it:Levy_strategy-differentiability} For each $z\in\C^+$, the function $\G_{\boldsymbol{\cdot}}(z): [0,T] \to \C^-, t \mapsto \G_t(z)$ is continuously differentiable; we denote its derivative by $\dot{\G}_{\boldsymbol{\cdot}}(z): [0,T] \to \C^-$.
 \item\label{it:Levy_strategy-derivative_bound} There are $k\in\N$ and $\theta\geq 0$ such that 
$$\big|\dot{\G}_t(z)\big| \leq - \frac{\Im(\G_t(z))}{\Im(z)^k}\, \theta \qquad\text{for all $(t,z) \in [0,T] \times \C^+$}.$$
\end{enumerate}
Then, for every $\epsilon>0$, it holds true that
\begin{equation}\label{eq:Levy_strategy_integral-bound}
\frac{1}{\pi} \int^\infty_{-\infty} \big| \G_{\rho_T}(s+i\epsilon) - \G_{\rho_0}(s+i\epsilon) \big|\, ds \leq \frac{T \theta}{\epsilon^k},
\end{equation}
and with the absolute constant $c_k := (2k+1)(\frac{1}{k^2 \pi})^{\frac{k}{2k+1}}$, we have that
\begin{equation}\label{eq:Levy_strategy_Levy-bound}
L(\mu_{\rho_T},\mu_{\rho_0}) \leq c_k T^{\frac{1}{2k+1}} \theta^{\frac{1}{2k+1}}.
\end{equation}
\end{lemma}

\begin{proof}
Take any $\epsilon>0$. Thanks to condition \ref{it:Levy_strategy-differentiability}, the fundamental theorem of calculus allows us to write $\G_{\rho_T}(s+i\epsilon) - \G_{\rho_0}(s+i\epsilon) = \int^T_0 \dot{\G}_t(s+i\epsilon)\, dt$ for every $s\in \R$. In combination with the bound required in condition \ref{it:Levy_strategy-derivative_bound}, this yields that
$$\big| \G_{\rho_T}(s+i\epsilon) - \G_{\rho_0}(s+i\epsilon) \big|
\leq - \frac{\theta}{\epsilon^k} \int^T_0 \Im(\G_t(s + i \epsilon))\, dt$$
for all $s\in\R$. By integrating the latter inequality over $s\in\R$ and using the Fubini-Tonelli theorem, which can be applied thanks to condition \ref{it:Levy_strategy-continuity}, we conclude that
$$\frac{1}{\pi} \int^\infty_{-\infty} \big| \G_{\rho_T}(s+i\epsilon) - \G_{\rho_0}(s+i\epsilon) \big|\, ds
\leq \frac{\theta}{\epsilon^k} \int^T_0 \int^\infty_{-\infty} \Big(-\frac{1}{\pi} \Im(\G_t(s + i \epsilon)) \Big)\, ds\, dt
= \frac{T \theta}{\epsilon^k}.$$
Notice that the last step relies on the fact that $-\frac{1}{\pi} \Im(\G_t(s + i \epsilon))\, ds$ is a probability measure on $\R$ for each $t\in[0,T]$. This verifies the bound asserted in \eqref{eq:Levy_strategy_integral-bound}. 
Finally, we involve the estimate for the L\'evy distance given in Theorem \ref{thm:Levy_distance}. The already established bound \eqref{eq:Levy_strategy_integral-bound} gives us that $L(\mu_{\rho_T},\mu_{\rho_0}) \leq 2\sqrt{\frac{\epsilon}{\pi}} + \frac{T \theta}{\epsilon^k}$ for all $\epsilon>0$. Minimizing the right hand side over $\epsilon \in (0,\infty)$ yields the asserted bound \eqref{eq:Levy_strategy_Levy-bound}.
\end{proof}

With the help of Lemma \ref{lem:Levy_strategy}, we can finally prove Theorem \ref{thm:Levy_Hoelder}. Notice that the uniform continuity of $\mu: (B_\sa \times \P_2(B),d) \to (\Prob(\R),L)$ with respect to any metric $d$ on $B_\sa \times \P_2(B)$ which yields a metrization of the product topology is an immediate consequence of the bounds stated in the items \ref{it:Levy_Hoelder-1} and \ref{it:Levy_Hoelder-2} of that theorem. To see this, we only need to notice that, for any two data pairs $\rho_0 = (b_{0,0},\eta_0)$ and $\rho_1 = (b_{0,1},\eta_1)$ in $B_\sa \times \P_2(B)$, we have that
$$L(\mu_{\rho_1},\mu_{\rho_0})
\leq L(\mu_{(b_{0,1},\eta_1)},\mu_{(b_{0,0},\eta_1)}) + L(\mu_{(b_{0,0},\eta_1)},\mu_{(b_{0,0},\eta_0)})
\leq c_1\, \|b_{0,1}-b_{0,0}\|^{1/3} + c_2\, \|\eta_1-\eta_0\|^{1/5}.$$
Thus, it suffices to prove \ref{it:Levy_Hoelder-1} and \ref{it:Levy_Hoelder-2} in Theorem \ref{thm:Levy_Hoelder}; we treat these two parts separately.

\begin{proof}[Proof of Theorem \ref{thm:Levy_Hoelder} \ref{it:Levy_Hoelder-1}]
For fixed $b_0 \in B_\sa$ and $\eta_0,\eta_1 \in \P_2(B)$, we consider the $C^1$-path
$$\eta:\ [0,1] \to \P_2(B),\quad t \mapsto \eta_t := (1-t) \eta_0 + t \eta_1$$
and the family of data pairs $(\rho_t)_{t\in[0,1]}$ given by $\rho_t := (b_0,\eta_t)$ for $t\in [0,1]$. In this case, the function $\G: [0,1] \times \C^+ \to \C^-, (t,z) \mapsto \G_t(z)$ defined by $\G_t(z) := \G_{\rho_t}(z)$ for $(t,z) \in [0,1] \times \C^+$ is related with the function $G: [0,1] \times \bH^+(B) \to \bH^-(B), (t,b) \mapsto G_t(b)$ defined by $G_t(b) := G_{\eta_t}(b)$ for all $(t,z) \in [0,1] \times \bH^+(B)$ via $\G_t(z) = \phi(G_t(z\1-b_0))$ for all $(t,z) \in [0,1] \times \C^+$. 

From Theorem \ref{thm:continuity}, we derive that $G: [0,1] \times \bH^+(B) \to \bH^-(B)$ is continuous, and in Theorem \ref{thm:Burgers}, we have seen that $G_{\boldsymbol{\cdot}}(b): [0,1] \to \bH^-(B)$ is continuously differentiable for every fixed $b\in \bH^+(B)$; we infer that $\G$ satisfies the conditions formulated in Item \ref{it:Levy_strategy-continuity} and Item \ref{it:Levy_strategy-differentiability} of Lemma \ref{lem:Levy_strategy}. To verify also the condition formulated in Item \ref{it:Levy_strategy-derivative_bound}, we proceed as follows. First, we deduce from Theorem \ref{thm:Burgers} that $G$ solves the partial differential equation
$$\dot{G}_t(b) = - (D G_t)(b)\big((\eta_1-\eta_0)(G_t(b))\big) \qquad\text{for $(t,b) \in [0,1] \times \bH^+(B)$}.$$
By applying $\phi$ to both sides, we derive with the bound \eqref{eq:Cauchy_derivatives_bound-2} provided by Lemma \ref{lem:Cauchy_derivatives_bound} that
\begin{multline*}
\big|\phi(\dot{G}_t(b))\big| \leq - \Im(\phi(G_t(b))) \|\Im(b)^{-1}\| \|(\eta_1-\eta_0)\big( G_t(b) \big)\|\\
                             \leq - \Im(\phi(G_t(b))) \|\Im(b)^{-1}\|^2 \|\eta_1-\eta_0\|.
\end{multline*}
When applied for $b = z\1 - b_0$, the latter says that for $\G$ the condition \ref{it:Levy_strategy-derivative_bound} is satisfied with
\begin{equation}\label{eq:Cauchy_derivatives_bound-Levy_Hoelder-1}
\big|\dot{\G}_t(z)\big| \leq - \frac{\Im(\G_t(z))}{\Im(z)^2}\, \|\eta_1-\eta_0\| \qquad\text{for $(t,z) \in [0,1] \times \C^+$}.
\end{equation}
Thus, \eqref{eq:Levy_strategy_Levy-bound} in Lemma \ref{lem:Levy_strategy} yields the asserted bound $L(\mu_{\rho_1},\mu_{\rho_0}) \leq c_2 \|\eta_1-\eta_0\|^{1/5}$.
\end{proof}

\begin{proof}[Proof of Theorem \ref{thm:Levy_Hoelder} \ref{it:Levy_Hoelder-2}]
Suppose that $b_{0,0},b_{0,1}\in B_\sa$ and $\eta \in \P_2(B)$ are given. We set
$$b_0:\ [0,1] \to B_\sa,\quad t \mapsto b_{0,t} := (1-t) b_{0,0} + t b_{0,1}$$
and consider the family of data pairs $(\rho_t)_{t\in[0,1]}$ given by $\rho_t := (b_{0,t},\eta)$ for $t\in [0,1]$.
We define $G: [0,1] \times \bH^+(B) \to \bH^-(B), (t,b) \mapsto G_t(b)$ by $G_t(b) := G_\eta(b-b_{0,t})$ for all $(t,b)\in [0,1] \times \bH^+(B)$. The function $\G: [0,1] \times \C^+ \to \C^-, (t,z) \mapsto \G_t(z)$ defined by $\G_t(z) := \G_{\rho_t}(z)$ for $(t,z) \in [0,1] \times \C^+$ is then related with $G$ via $\G_t(z) = \phi(G_t(z\1))$ for all $(t,z) \in [0,1] \times \C^+$. 

Clearly, $G: [0,1] \times \bH^+(B) \to \bH^-(B)$ is continuous and $G_{\boldsymbol{\cdot}}(b): [0,1] \to \bH^-(B)$ is continuously differentiable with derivative given by $\dot{G}_t(b) = -(D G_\eta)(b-b_{0,t})(b_{0,1}-b_{0,0})$ for every fixed $b \in \bH^+(B)$; hence, $\G$ satisfies the conditions formulated in Item \ref{it:Levy_strategy-continuity} and Item \ref{it:Levy_strategy-differentiability} of Lemma \ref{lem:Levy_strategy}. To verify condition \ref{it:Levy_strategy-derivative_bound}, we use the bound \eqref{eq:Cauchy_derivatives_bound-2} provided by Lemma \ref{lem:Cauchy_derivatives_bound},
$$\big|\phi(\dot{G}_t(b))\big| \leq - \Im(\phi(G_t(b))) \|\Im(b)^{-1}\| \|b_{0,1}-b_{0,0}\|,$$
which yields, when applied for $b=z\1$, that
\begin{equation}\label{eq:Cauchy_derivatives_bound-Levy_Hoelder-2}
\big|\dot{\G}_t(z)\big| \leq - \frac{\Im(\G_t(z))}{\Im(z)}\, \|b_{0,1}-b_{0,0}\| \qquad\text{for $(t,z) \in [0,1] \times \C^+$}.
\end{equation}
Thus, \eqref{eq:Levy_strategy_Levy-bound} in Lemma \ref{lem:Levy_strategy} yields the asserted bound $L(\mu_{\rho_1},\mu_{\rho_0}) \leq c_1 \|b_{0,1}-b_{0,0}\|^{1/3}$.
\end{proof}

We conclude by the following consequence of the proof of Theorem \ref{thm:Levy_Hoelder}.

\begin{corollary}\label{cor:integral_bounds}
In the situation of Theorem \ref{thm:Levy_Hoelder}, the following bounds hold true:
\begin{enumerate}
 \item\label{it:integral_bounds-1} For each choice of $\eta \in \P_2(B)$, we have for all $b_{0,0},b_{0,1} \in B_\sa$ and $\epsilon>0$ that
$$\frac{1}{\pi} \int^\infty_{-\infty} \big| \G_{\mu_{(b_{0,1},\eta)}}(s+i\epsilon) - \G_{\mu_{(b_{0,0},\eta)}}(s+i\epsilon) \big|\, ds \leq \frac{1}{\epsilon}\, \|b_{0,1} - b_{0,0}\|.$$
 \item\label{it:integral_bounds-2} For each choice of $b_0 \in B_\sa$, we have for all $\eta_0,\eta_1\in\P_2(B)$ and $\epsilon>0$ that
$$\frac{1}{\pi} \int^\infty_{-\infty} \big|\G_{\mu_{(b_0,\eta_1)}}(s+i\epsilon) - \G_{\mu_{(b_0,\eta_0)}}(s+i\epsilon)\big|\, ds \leq \frac{1}{\epsilon^2}\, \|\eta_1 - \eta_0\|.$$
\end{enumerate}
\end{corollary}

\begin{proof}
Using \eqref{eq:Levy_strategy_integral-bound} from Lemma \ref{lem:Levy_strategy}, we obtain the bounds asserted in the corollary from the bounds \eqref{eq:Cauchy_derivatives_bound-Levy_Hoelder-1} and \eqref{eq:Cauchy_derivatives_bound-Levy_Hoelder-2} which we derived in the proofs of \ref{it:Levy_Hoelder-1} and \ref{it:Levy_Hoelder-2} in Theorem \ref{thm:Levy_Hoelder}.
\end{proof}

\begin{appendix}

\section{Free noncommutative function theory on the second level}\label{sec:FreeNCFunctionTheory}

Free noncommutative function theory has its origins in the work of Taylor \cite{Taylor1972,Taylor1973}, whose ideas were taken up subsequently by numerous authors for various applications; we mention \cite{Voi2004,Voi2008,Voi2010} for their relevance on free probability theory and since Proposition \ref{prop:opval_semicircular_Cauchy_matrix} fits nicely into that context. A unified and very general approach was presented by Kaliuzhnyi-Verbovetskyi and Vinnikov in \cite{KVV2014}, building on the notion of noncommutative sets and functions. With an eye towards applications in the concrete situation of this paper, we introduce here the terminology of noncommutative sets and functions in some truncated version. For the sake of simplicity, we restrict ourselves to the case where those are built on $C^\ast$-algebras.

Let $A$ be a unital $C^\ast$-algebra. For any $n\in \N$, we refer to $M_{\leq n}(A) := \coprod_{k=1}^n M_k(A)$ as the \emph{noncommutative space over $A$ of height $n$}. For a subset $\Omega \subseteq M_{\leq n}(A)$, we call $\Omega_k := M_k(A) \cap \Omega$ for $1\leq k \leq n$ the \emph{$k$-th level of $\Omega$}. We say that $\Omega$ is a \emph{noncommutative set over $A$ of height $n$} if $\Omega$ is closed under direct sums of appropriate size, i.e., we require that whenever $1 \leq k,l \leq n$ are such that $k+l \leq n$, then 
$$X \oplus Y :=  \begin{bmatrix} X & 0 \\ 0 & Y \end{bmatrix} \in \Omega_{k+l}$$
for all $X \in \Omega_k$ and $Y \in \Omega_l$.
We say that $\Omega$ is \emph{open} if each level $\Omega_k$ is open in the operator-norm topology of the respective $C^\ast$-algebra $M_k(A)$.

Now, consider two unital $C^\ast$-algebras $A$ and $B$ and let $\Omega \subseteq M_{\leq n}(A)$ be a noncommutative set of height $n$. Consider a function $f^{(\leq n)}: \Omega \to M_{\leq n}(B)$. We shall denote by $f^{(k)}: \Omega_k \to M_{\leq n}(B)$ the restriction of $f^{(\leq n)}$ to the $k$-th level $\Omega_k$ of $\Omega$. We call $f^{(\leq n)}$ a \emph{noncommutative function of height $n$} if $f^{(k)}(\Omega_k) \subseteq M_k(B)$ for $k=1,\dots,n$ and the following two conditions are satisfied:
\begin{itemize}
 \item \emph{$f^{(\leq n)}$ respects direct sums}, i.e., whenever $1 \leq k,l \leq n$ are such that $k+l \leq n$, then $f^{(k+l)}(X \oplus Y) = f^{(k)}(X) \oplus f^{(l)}(Y)$ holds true for all $X \in \Omega_k$ and $Y \in \Omega_l$.
 \item \emph{$f^{(\leq n)}$ respects similarities}, i.e., if $X\in \Omega_k$ with $1\leq k \leq n$ and an invertible matrix $S\in M_k(\C)$ are given such that $S X S^{-1} \in \Omega_k$, then $f^{(k)}(S X S^{-1}) = S f^{(k)}(X) S^{-1}$.
\end{itemize}
We say that $f$ is \emph{continuous} if $f^{(k)}: \Omega_k \to M_k(B)$ is continuous for $k=1,\dots,n$ in the respective operator-norm topologies. 

A fundamental fact in free noncommutative function theory is that noncommutative functions, due to their algebraic structure, show strong analytic properties under rather weak assumptions; see \cite[Theorem 7.4]{KVV2014}. To some extent, this remains true in the truncated setting. The following theorem, which is a variant of \cite[Proposition 2.5]{HKMcC2011} in the spirit of \cite[Lemma 3.1]{AMcC2015}, makes this claim precise. In essence, it says that continuous functions between subsets of $C^\ast$-algebras are automatically (Fr\'echet) holomorphic if they admit an extension to a noncommutative function of height $2$.

\begin{theorem}\label{theo:matrix-derivative}
Let $A$ and $B$ be unital $C^\ast$-algebras and let $\emptyset \neq \Omega_1 \subseteq A$ be an open subset. Consider a continuous function $f^{(1)}: \Omega_1 \to B$ with the following property: there exists an open noncommutative set $\Omega$ over $A$ of height $2$ and a continuous noncommutative function $f^{(\leq 2)}: \Omega \to M_{\leq 2}(B)$ of height $2$ which extends $f^{(1)}: \Omega_1 \to B$ in the sense that $\Omega_1$ is the first level set of $\Omega$ and $f^{(1)}$ is the restriction of $f^{(\leq 2)}$ to $\Omega_1$.
Then $f^{(1)}$ is holomorphic and, for all $a \in \Omega_1$ and $h \in A$ satisfying the condition
\begin{equation}\label{eq:matrix-condition}
\begin{bmatrix} a & h\\ 0 & a\end{bmatrix} \in \Omega_2,
\end{equation}
it holds true that
\begin{equation}\label{eq:matrix-derivative}
f^{(2)}\bigg(\begin{bmatrix} a & h\\ 0 & a \end{bmatrix} \bigg) = \begin{bmatrix} f^{(1)}(a) & (D f^{(1)})(a) h \\ 0 & f^{(1)}(a) \end{bmatrix}.
\end{equation}
\end{theorem}

Before we give the proof, we notice that for $a_1,a_2 \in A$ and $z\in\C$ the matrix identity
\begin{equation}\label{eq:intertwining}
\begin{bmatrix} a_1 & z(a_2-a_1)\\ 0 & a_2 \end{bmatrix} \begin{bmatrix} 1 & z\\ 0 & 1 \end{bmatrix} = \begin{bmatrix} 1 & z\\ 0 & 1 \end{bmatrix} \begin{bmatrix} a_1 & 0 \\ 0 & a_2 \end{bmatrix}
\end{equation}
holds; this provides similarity relations on $M_2(A)$, which will be used in the sequel.

\begin{proof}[Proof of Theorem \ref{theo:matrix-derivative}]
Take $a \in \Omega_1$ and $h \in A$ such that \eqref{eq:matrix-condition} is satisfied. Since $\Omega_1$ and $\Omega_2$ are assumed to be open, we find $r>0$ such that
$$a + z h \in \Omega_1 \qquad\text{and}\qquad \begin{bmatrix} a & h\\ 0 & a + z h\end{bmatrix} \in \Omega_2$$
for all $z\in\C$ satisfying $0 < |z| < r$. For every such $z$, notice that \eqref{eq:intertwining} gives
$$\begin{bmatrix} a & h\\ 0 & a + z h\end{bmatrix} \begin{bmatrix} 1 & \frac{1}{z}\\ 0 & 1 \end{bmatrix} = \begin{bmatrix} 1 & \frac{1}{z}\\ 0 & 1 \end{bmatrix} \begin{bmatrix} a & 0 \\ 0 & a + z h\end{bmatrix},$$
from which we infer that
$$f^{(2)}\bigg( \begin{bmatrix} a & h\\ 0 & a + z h\end{bmatrix} \bigg) = \begin{bmatrix} 1 & \frac{1}{z}\\ 0 & 1 \end{bmatrix} \begin{bmatrix} f^{(1)}(a) & 0 \\ 0 & f^{(1)}(a + z h) \end{bmatrix} \begin{bmatrix} 1 & -\frac{1}{z}\\ 0 & 1 \end{bmatrix}.$$
By another application of \eqref{eq:intertwining}, we deduce from the latter that
$$f^{(2)}\bigg( \begin{bmatrix} a & h\\ 0 & a + z h\end{bmatrix} \bigg) = \begin{bmatrix} f^{(1)}(a) & \frac{1}{z}(f^{(1)}(a + z h) - f^{(1)}(a))\\ 0 & f^{(1)}(a + z h) \end{bmatrix}.$$
As both $f^{(1)}$ and $f^{(2)}$ are continuous, the previous identity tells us when letting $z\to 0$ that the limit $(\delta f^{(1)})(a;h) = \lim_{z\to 0} \frac{1}{z}(f^{(1)}(a + z h) - f^{(1)}(a))$ exists and that
\begin{equation}\label{eq:matrix-derivative_preliminary}
f^{(2)}\bigg(\begin{bmatrix} a & h\\ 0 & a \end{bmatrix} \bigg) = \begin{bmatrix} f^{(1)}(a) & (\delta f^{(1)})(a;h) \\ 0 & f^{(1)}(a) \end{bmatrix}.
\end{equation}
Because both $\Omega_1$ and $\Omega_2$ are supposed to be open, there exists for every given $a \in \Omega_1$ some $\epsilon>0$ such that \eqref{eq:matrix-condition} is satisfied for all $h\in A$ with $\|h\|< \epsilon$. Thus, it follows that $f^{(1)}$ is G\^ateaux holomorphic, and because it is continuous by assumption and so in particular locally bounded, we conclude that $f^{(1)}$ is holomorphic; see the discussion in Section \ref{subsec:holomorphy}. In particular, we may replace in \eqref{eq:matrix-derivative_preliminary} the G\^ateaux derivative $\delta f^{(1)}(a;h)$ by the Fr\'echet derivative $(D f^{(1)})(a)h$ which leads us to the formula \eqref{eq:matrix-derivative}.
\end{proof}

For a function $f^{(\leq n)}: \Omega \to M_{\leq n}(B)$ with $f^{(k)}(\Omega_k) \subseteq M_k(B)$ for $k=1,\dots,n$, we say that \emph{$f^{(\leq n)}$ respects intertwinings} if it satisfies the following condition: whenever $X \in \Omega_k$, $Y\in \Omega_l$, and $T \in M_{k \times l}(\C)$ for $1 \leq k,l \leq n$ are such that $X T = T Y$, then $f^{(k)}(X) T = T f^{(l)}(Y)$. Like in \cite[Proposition 2.1]{KVV2014}, one sees that if $f^{(\leq n)}$ respects intertwinings, then it respects direct sums and similarities, and is thus a noncommutative function of height $n$; the proof of the converse implication, however, does not work in the truncated setting due to the limitation of height. It is therefore interesting to note that Proposition \ref{prop:opval_semicircular_Cauchy_matrix} can be strengthened:

\begin{proposition}\label{prop:opval_semicircular_Cauchy_matrix_revisited}
Let $\eta: B \to B$ be an $n$-positive linear map on the unital $C^\ast$-algebra $B$. Then the noncommutative function $G^{(\leq n)}_\eta: \bH^+(M_{\leq n}(B)) \to \bH^-(M_{\leq n}(B))$ of height $n$ introduced in Proposition \ref{prop:opval_semicircular_Cauchy_matrix} respects intertwinings.
\end{proposition}

The proof of Proposition \ref{prop:opval_semicircular_Cauchy_matrix_revisited} is similar to that of Proposition \ref{prop:opval_semicircular_Cauchy_matrix} and is thus omitted.

To some extent, one can imitate in such cases the construction of the left and right difference-differential operators $\Delta_L$ and $\Delta_R$ from \cite[Section 2.2]{KVV2014}. We focus on $\Delta_R$ since $\Delta_L$ can be treated analogously. Let $f^{(\leq n)}: \Omega \to M_{\leq n}(B)$ be a noncommutative function of height $n$ on an open noncommutative set $\Omega \subseteq M_{\leq n}(A)$ of height $n$ which respects intertwinings. In the same way as \cite[Proposition 2.4]{KVV2014}, one proves that if $X \in \Omega_k$ and $Y \in \Omega_l$ are given for $1 \leq k,l\leq n$ with $k+l\leq n$, then there exists a $\C$-homogeneous map $\Delta_R f^{(\leq n)}(X,Y): M_{k \times l}(A) \to M_{k \times l}(B)$ such that
\begin{equation}\label{eq:difference-differential-operator}
f^{(k+l)}\bigg(\begin{bmatrix} X & Z\\ 0 & Y \end{bmatrix}\bigg) = \begin{bmatrix} f^{(k)}(X) & \Delta_R f^{(\leq n)}(X,Y)(Z)\\ 0 & f^{(l)}(Y) \end{bmatrix} \quad\text{whenever}\quad \begin{bmatrix} X & Z\\ 0 & Y \end{bmatrix} \in \Omega_{k+l}.
\end{equation}
While the proof of \cite[Proposition 2.6]{KVV2014} does not work without imposing further constraints on $k$ and $l$ due to the limitation of height, the first order difference formula from \cite[Theorem 2.11]{KVV2014} extends to the truncated setting with exactly the same proof: if $X \in \Omega_k$, $Y \in \Omega_l$, and $S \in M_{k \times l}(\C)$ are given for $1 \leq k,l\leq n$ with $k+l\leq n$, then
\begin{equation}\label{eq:difference-formula}
S f^{(k)}(Y) - f^{(l)}(X) S = \Delta_R f^{(\leq n)}(X,Y)(SY-XS).
\end{equation}
In the situation of Theorem \ref{theo:matrix-derivative} and under the additional assumption that $f^{(\leq 2)}: \Omega \to M_{\leq 2}(B)$ respects intertwinings, one can easily deduce from \eqref{eq:difference-formula} that $\Delta_R f^{(\leq 2)}(a,a)(h) = D f^{(1)}(a) h$ for all $a\in \Omega_1$ and $h \in A$; thus, \eqref{eq:matrix-condition} is recovered by applying \eqref{eq:difference-differential-operator}.

\begin{remark}\label{rem:Lipschitz_bound_details}
It is worthwhile to note that with the aforementioned facts, the alternative derivation of Lemma \ref{lem:Lipschitz_bound} which we sketched after the detailed proof becomes more transparent; in fact, we see that \eqref{eq:intertwining} was a shortcut avoiding the difference-differential operator. For $\eta \in \P_2(B)$, we can derive from \eqref{eq:difference-differential-operator} with the help of Lemma \ref{lem:half-plane} like in the proof of Lemma \ref{lem:Cauchy_derivatives_bound} that $\|\Delta_R G_\eta^{(\leq 2)}(b_0,b_1)(w)\| \leq \|\Im(b_0)^{-1}\| \|\Im(b_1)^{-1}\| \|w\|$ for all $b_0,b_1 \in \bH^+(B)$ and every $w \in B$; this result is in accordance with \cite[Proposition 3.1]{Belinschi2017}.
Further, we infer from the first order difference formula \eqref{eq:difference-formula} that $G_\eta(b_1) - G_\eta(b_0) = \Delta_R G_\eta^{(\leq 2)}(b_0,b_1)(b_1-b_0)$ for all $b_0,b_1 \in \bH^+(B)$. By combining both observations, we arrive again at the bound \eqref{eq:Lipschitz_bound}; we remark that this strategy is similar to the proof of \cite[Proposition 3.1]{Belinschi2017}.
\end{remark}

\end{appendix}


\bibliographystyle{abbrv}
\bibliography{ref-Dyson_equation_two_positive}

\end{document}